\title{Suppression of Chemotactic Singularity via Viscous Flow with Large Buoyancy}
\author{Zhongtian Hu\thanks{
                  Department of
Mathematics, Duke University, Durham, NC 27708, USA; email: zhongtian.hu@duke.edu}
}
\documentclass[11pt]{article}
\usepackage[margin=1in]{geometry}
\usepackage{comment,slashed,tensor}
\usepackage[T1]{fontenc}
\usepackage{ stmaryrd }
\usepackage[symbol]{footmisc}

\usepackage[bookmarksnumbered, bookmarksopen, colorlinks, citecolor=red, linkcolor=black]{hyperref}
 \usepackage{amsfonts}
 \usepackage{comment}
  \usepackage{cancel}
\usepackage{graphicx}
 \usepackage{amsmath, amssymb}
 \usepackage{amsthm}
 \usepackage{mathtools}

 \numberwithin{equation}{section}
\newtheorem{thm}{Theorem}[section]
\newtheorem{lem}[thm]{Lemma}
\newtheorem{cor}[thm]{Corollary}

\newtheorem{prop}[thm]{Proposition}
\newtheorem{rmk}[thm]{Remark}
\newtheorem{defn}[thm]{Definition}

\newcommand{\R}{\mathbb{R}}
\newcommand{\N}{\mathbb{N}}

\newcommand{\T}{\mathbb{T}}

\newcommand{\calB}{B_0}

\newcommand{\calT}{\mathcal{T}}
\newcommand{\calC}{\mathcal{C}}

\newcommand{\brho}{\bar\rho}
\newcommand{\trho}{\tilde{\rho}}

\newcommand{\mfg}{\mathfrak{g}}

\makeatletter
\newcommand*{\rom}[1]{\expandafter\@slowromancap\romannumeral #1@}
\makeatother

\DeclareMathOperator{\divv}{div}

\DeclareMathOperator{\ra}{Ra}
\DeclareMathOperator{\re}{Re}

\newcommand{\p}{\partial}

 \newcommand{\red}[1]{{#1}}
\begin{document}
%\onehalfspacing
\newpage
\maketitle
%%%%%%%%%%%%%%%%%%%%%%%%%%%
% abstract, keywords and Subject classification are optional.
%%%%%%%%%%%%%%%%%%%%%%%%%%%
\begin{abstract}
    In this work, we study the Keller-Segel-Navier-Stokes equation with low Reynolds number and subject to large buoyancy force. We show that for initial cell density with arbitrarily large mass (i.e. the $L^1$ norm), the solution remains regular for all times in the regime of sufficiently large buoyancy and viscosity. The major blowup suppression mechanism is a norm-stabilizing property possessed by a ``static problem,'' where the full problem can be seen as a perturbation of this quasi-stationary model. 
\end{abstract}

%%%%%%%%%%%%%%%%%%%%%%
% % Here is the start of the Text
%%%%%%%%%%%%%%%%%%%%%%
\renewcommand{\thefootnote}{\roman{footnote}}
\section{Introduction}
We consider the parabolic-elliptic Keller-Segel equation in {a} periodic channel $\Omega = \T \times [0,\pi]$ subject to the influence of {the} buoyancy-driven Navier-Stokes equation:
\begin{equation}
\label{eq:ksgen}
    \begin{cases}
        \p_t \rho + u\cdot \nabla \rho -\Delta \rho + \divv(\rho\nabla(-\Delta_N)^{-1}(\rho- \rho_m)) = 0,\\
        \rho_m = \frac{1}{|\Omega|}\int_\Omega \rho(t,x) dx,\\
        \p_t u + u\cdot\nabla u - \frac{1}{\re}\Delta u + \nabla p = \ra\rho (0,1)^T,\\
        \divv u = 0.
    \end{cases}
\end{equation}
We equip the system with initial data $\rho(0,x) = \rho_0(x)$, $u(0,x) = u_0(x)$, where $\rho_0$ {is a nonnegative scalar function} and $u_0$ is a divergence-free vector field. We also consider the following set of boundary conditions:
\begin{equation}
    \label{bc}
    \p_2 \rho = \nabla \rho \cdot n = 0,\quad u_2 = u\cdot n = 0,\quad \omega = \nabla^\perp\cdot u = 0,\quad\quad \text{on $\p\Omega$}.
\end{equation}
Here, $\T := [-\pi, \pi)$ is the one dimensional torus, and a function $f$ defined on $\T$ means that $f$ assumes {the} periodic boundary condition with period $2\pi$; $n = (0,1)^T$ denotes the unit normal derivative along $\p\Omega = \T \times \{0,\pi\}$; {$\nabla^\perp$ denotes the differential operator $ (-\p_2, \p_1)$.}

\noindent The first equation in \eqref{eq:ksgen} is {the classical parabolic--elliptic Keller-Segel equation with advection.} This equation characterizes {a population of bacteria with the density} $\rho$ that moves in response to {an} attractive chemical that the bacteria themselves secrete. {Specifically, the chemical-induced aggregation effect is modeled by the term $\divv(\rho\nabla (-\Delta_N)^{-1}(\rho-\rho_m))$, indicating a scenario where chemicals homogenize much faster than the motion of micro-organisms. Here, $-\Delta_N$ denotes homogeneous Neumann Laplacian that represents the classical condition of zero chemical flux across boundary $\p\Omega$.} Furthermore, chemotaxis usually takes place in ambient viscous fluids, which are classically modeled by Navier-Stokes equations with velocity $u$ and pressure $p$. In nature, the micro-organisms and the fluid can interact through various means (see, e.g., \cite{he2023enhanced}). The main interaction on which we focus is the coupling by buoyancy, {which originates from the variation of bacterial density in the domain.} Mathematically, such interaction appears in the fluid equation through the forcing term $\ra \rho (0,1)^T$, where $\ra$ denotes the Rayleigh number {measuring relative buoyancy strength due to density variation.}

{In this work, we focus on the study of chemotaxis-fluid interaction in a canonical domain $\T \times [0,\pi]$. {We remark that previous studies classically consider domains without boundaries, e.g. $\R^2, \T^2, \T\times \R$, in order to study chemotaxis-fluid interaction in the bulk of fluid (see \cite{bedrossian2017suppression,he2018suppression,kiselev2016suppression,he2023enhanced}).} Our work, in contrast, intends to investigate the {buoyancy effect} along the vertical direction. From this perspective, it is interesting to consider a domain with a top and bottom boundary, and study the chemotaxis-fluid interaction via buoyancy when taking boundary effects into account.} We also emphasize the boundary conditions that we insist for the rest of this work: on one hand, we assert the classical homogeneous Neumann boundary condition for the cell density. On the other hand, we impose the Lions boundary condition onto the fluid equation: $u\cdot n|_{\p \Omega} = 0,\; \omega|_{\p \Omega} = 0$. This particular boundary condition, as a special case of the more general Navier boundary condition, is prevalent in simulations of flows in the presence of rough boundaries, such as in hemodynamics. It was first introduced and rigorously studied by J.-L. Lions \cite{lions1969quelques} and P.-L. Lions \cite{lions1997mathematical}. We refer the readers to works such as \cite{kelliher2006navier,filho2005inviscid,clopeau1998vanishing,xiao2007vanishing} and references therein for a more thorough discussion regarding the Navier boundary condition.

When the ambient fluid is absent (i.e., when $\ra = 0$ and $u\equiv 0$ in \eqref{eq:ksgen}), we recover the classical parabolic-elliptic Keller-Segel equation:
\begin{equation}
    \label{eq:ks}
    \p_t \rho - \Delta \rho + \divv(\rho\nabla(-\Delta_N)^{-1}(\rho - \rho_m)) = 0,\quad \text{ in }\Omega.
\end{equation}
First introduced by Patlak \cite{patlak1953random}, and Keller and Segel \cite{keller1971model}, \eqref{eq:ks} has been classically studied in various settings. We refer the interested readers to the following list of works: \cite{bedrossian2015large,bedrossian2014existence,biler20068pi,blanchet2008infinite,blanchet2006two,calvez2008parabolic,carrillo2008uniqueness,horstmann20031970,horstmann20041970,jager1992explosions,nagai1995blow,nagai1997application}. A remarkable feature enjoyed by \eqref{eq:ks} is that the solution can form singularity in finite time when dimension is greater than $1$. In dimension $2$, \eqref{eq:ks} is $L^1$-critical. For any initial datum $\rho_0$ with finite second moment, the solution is globally regular if the initial mass $\|\rho_0\|_{L^1} < 8\pi$, see e.g. \cite{bedrossian2014existence,blanchet2006two,carrillo2008uniqueness,jager1992explosions,wei2018global}. If the initial mass is strictly greater than $8\pi$, a finite-time singularity forms, as seen in \cite{blanchet2006two,calvez2008parabolic,jager1992explosions,nagai1995blow}. A more careful analysis of such blowup solutions are also carried out \cite{collot2022refined,collot2022spectral,velazquez2002stability,velazquez2004point,velazquez2004point2}.

It is also curious to understand the behavior of Keller-Segel equation under the influence of fluid advection, given the fact that most chemotactic processes take place in ambient fluid. {The presence of fluid advection can bring complicated effects to chemotaxis. Among these effects, we would like to focus on the \textit{regularization effects} induced by fluid advection. That is, we would like to understand how the transport term $u\cdot\nabla \rho$ could prevent potential singular behaviors in \eqref{eq:ks}. For the past decade, much progress has been made to understand such regularization effects in the context of Keller-Segel equation.}
%Besides possibly complicated effects induced by the fluid, how regularizing effect of an incompressible fluid advection, namely the term $u \cdot \nabla \rho$, can suppress chemotactic singularity is under active investigation recently. 
In the regime of passive advection (i.e. the fluid velocity $u$ is given and is not coupled to the Keller-Segel equation), Kiselev and Xu in \cite{kiselev2016suppression} first demonstrate that given any initial datum $\rho_0$, there exists a relaxation-enhancing flow (see \cite{constantin2008diffusion} for a precise definition) with sufficiently large amplitude that can suppress the singularity formation. This result is later generalized by \cite{iyer2021convection} to a larger class of passive flows and more general aggregation equations. Moreover, in \cite{bedrossian2017suppression,he2018suppression}, the authors exploit the enhanced dissipation phenomenon induced by strong monotone shear flows. Such flows effectively reduce the dimensionality of the problem, where in 2D the singularity can be suppressed. The fast-splitting scenario induced by hyperbolic flows are also  explored in \cite{he2019suppressing,he2022fast}.

There also have been numerous attempts to investigate the regularity properties of the Keller-Segel equation coupling to active fluid models, and many of which address the global regularity of solutions to such coupled systems. We highlight that, among those results, either there is smallness assumption on initial data (e.g. \cite{chae2014global,di2010chemotaxis,lorz2012coupled,duan2010global}), or the global regularity of both the chemotaxis equation and fluid equation still hold if they are uncoupled (e.g. \cite{winkler2012global,winkler2021suppressing}). We also note that the authors in \cite{he2023enhanced,zeng2021suppression} study the blowup suppression mechanism of Keller-Segel-Navier-Stokes equation near a strong Couette flow. {These results are almost linear in a sense that the main driven mechanism is still brought by a dominating passive background flow.}

Recently in a series of works \cite{hu2023suppression,hu2023stokes} by the author joining with Kiselev and Yao, they analyze how buoyancy effects in fluid equations suppress chemotactic singularities in a genuinely nonlinear setting. In \cite{hu2023suppression}, the authors investigated the Keller-Segel equation evolving in ambient porous media under the influence of buoyancy. The authors demonstrated that a coupling with the porous media equation via an arbitrarily weak buoyancy constant suffices to arrest any potential chemotactic blowup. The key argument in \cite{hu2023suppression} is a careful analysis of a potential energy and the coercive term $\|\p_1\rho\|_{H^{-1}_0}$ in its time derivative. The authors observed that this $H^{-1}_0$ norm has to be small, and induces an anisotropic mixing effect along $x_1$-direction. This effect renders the system quasi-one-dimensional and therefore suppresses finite-time blowup. On the other hand, \cite{hu2023stokes} studies how Stokes-Boussinesq flow with strong buoyancy suppresses blowup of a Keller-Segel equation equipped with zero Dirichlet boundary condition. A rather soft argument in \cite{hu2023stokes} shows that the flow quenches the $L^2$ norm of cell density to be sufficiently small in the regime of large buoyancy.

{The other motivation of studying the regularization by active fluid advection in the Keller-Segel equation is its connection to the global regularity of other equations that potentially have singular behaviors, such as some fluid equations. In fact, such regularizing phenomenon is the main factor that competes with main blowup mechanisms in some of the most classical fluid equations. For 3D Euler equations, a lack of/a weakened advection term is sufficient to induce finite-time singularities, see \cite{constantin1986note,elgindi2021finite}. The same is also observed in 3D Navier-Stokes equations, see \cite{lei2009stabilizing,hou2011singularity}. The crucial role played by advection stands out even more when one considers certain 1D models capturing Euler and Navier-Stokes dynamics, see \cite{chen2021regularity,jia2019gregorio}. The Keller-Segel equation much resembles these fluid equations since it is also a critical equation that exhibits potential singularities. Hence, studying the regularization by fluid advection in Keller-Segel equation contributes to a more substantial understanding in the problem of global existence for many fluid equations.}

{In this work, we plan to investigate \eqref{eq:ksgen} in the case where the fluid has large viscosity and exerts significant buoyancy onto the bacteria. We remark that not only micro-organisms generically live in viscous environment (see \cite{purcell1977life}), but also intriguing biological phenomena occur in a viscous fluid and in a buoyancy-dominating regime. For example, \cite{tuval2005bacterial} studies an interesting phenomenon called chemotactic Boycott effect in this regime, which describes the sedimentation process of both bacteria and fluid. In numerical experiments, the canonical choice of the coupling fluid is the Navier-Stokes flow with parameters $\ra\sim 10^6$ and $\re^{-1}\sim 10^3$. Hence, while the chemotaxis model used in \cite{tuval2005bacterial} is different from \eqref{eq:ksgen}, the parameter regime $\ra \gg \frac{1}{\re} \gg 1$ is still highly relevant in the study of chemotaxis in viscous environment. Moreover, from an analytical perspective, this is also a first step to showing the blowup suppression by Navier-Stokes flow with only large buoyancy and moderate viscosity.} 

To make the problem more amenable to analysis, it is convenient to write $B = \re^{-1}$ and $g = \ra\cdot\re$. Another handy transformation is that in periodic channel $\Omega$, one can reduce the Lions boundary condition in the following way: the no-flux boundary is equivalent to $u_2|_{\p_\Omega} = 0$. Since $\p_1$ is the tangential derivative on $\p\Omega$, we have the extra boundary constraint $\p_1^k u_2|_{\p\Omega} = 0$ for all $k \in \N$. \red{Thus, the Lions boundary condition can be equivalently written as
$$
u_2|_{\p\Omega} = 0,\quad \omega|_{\p\Omega} = \p_2 u_1|_{\p\Omega} = 0.
$$
}
To sum up, we will work with the following equivalent, yet more convenient form of \eqref{eq:ksgen}:
\begin{subequations}
\label{eq:ksgen2}
\begin{equation}
    \label{eq:ksstokesden2}
    \p_t \rho + u\cdot \nabla \rho -\Delta \rho + \divv(\rho\nabla(-\Delta_N)^{-1}(\rho- \rho_m)) = 0,\quad\text{ in }\Omega,
\end{equation}
\begin{equation}
    \label{eq:ksstokesvel2}
    \p_t u + u\cdot\nabla u - B\Delta u + \nabla p = Bg\rho (0,1)^T,\quad \divv u = 0,\quad\text{ in }\Omega,
\end{equation}
\begin{equation}
    \label{eq:lions}
    \p_2 \rho\big|_{\p \Omega} = 0,\; u_2|_{\p\Omega} = 0,\; \p_2 u_1|_{\p\Omega} = 0,
\end{equation}
\begin{equation}
    \label{eq:init}
    \rho(0,x) = \rho_0(x) \ge 0,\; u(0,x) = u_0(x).
\end{equation}
\end{subequations}

\subsection{Main Theorem}
The main result of this work is that the {\textit{regular solution}}\footnote{{Roughly speaking, the solution pair $(\rho,u)$ is called a \textit{regular solution} if it is $C^\infty$ in space and for positive times. We will give a rigorous definition in Section \ref{sec:lwp}.}} of \eqref{eq:ksgen2} with arbitrary large mass is in fact globally regular given both parameters $B$ and $g$ sufficiently large, whose sizes only depend on initial data $(\rho_0, u_0)$. The precise statement of the main result is given as follows:

\red{
\begin{thm}
    \label{thm:stokeswp}
    Suppose initial data $(\rho_0,u_0)$ with $\rho_0 \in H^1$ nonnegative and $u_0 \in V$, where $V$ is the class of $H^1$, divergence-free vector fields that satisfy no-flux boundary condition\footnote{The exact definition of $V$ will be given in Section \ref{subsect:lions}.}. There exists a couple $(g_1,B_1) = (g_1(\rho_m, \|\rho_0 - \rho_m\|_{L^2}), B_1(\rho_m, \|\rho_0 - \rho_m\|_{L^2}, \|u_0\|_1))$ such that if $g \ge g_1$, $B \ge B_1g^2e^{g^2}$, \eqref{eq:ksgen2} admits a unique, regular, and global-in-time solution.
\end{thm}
}

\begin{rmk}
    We compare this result with \cite{hu2023stokes}. Firstly, the two works treat different scenarios: \cite{hu2023stokes} investigates a Keller-Segel equation with zero Dirichlet boundary condition, while we study the Keller-Segel equation with classical zero Neumann boundary conditions. Secondly, the blowup-suppressing mechanism in the two works are different precisely due to the first difference above. In \cite{hu2023stokes}, the blowup is suppressed by an increased interaction between cells and the cold boundary induced by strong buoyancy force. This mechanism forces the $L^2$ norm of cell density to drop below the mass threshold in short time. In our work, we instead use a dimension reduction mechanism similar to that in \cite{hu2023suppression}. It stabilizes the $L^2$ norm of cell density around a possibly large number.
\end{rmk}

\subsection{Main Strategy of Proof}\label{subsect:strat}
In this section, we would like to first emphasize the main difficulty of proving Theorem \ref{thm:stokeswp}, after which we briefly sketch the idea of proving this main result.

The main obstacle is that, given the appearance of time derivative and advection term in the fluid equation \eqref{eq:ksstokesvel2}, the problem suffers from a loss of explicit velocity law, which directly connects density $\rho$ and fluid velocity $u$. In general, such loss prevents us from obtaining precise control of $u$ in a way similar to \cite{hu2023suppression}. In the parameter regime specified in Theorem \ref{thm:stokeswp}, however, we will show that \eqref{eq:ksgen2} is a suitable perturbation of a static problem \eqref{eq:ksstatstokes}, which actually has a  velocity law that can be analyzed with tools developed in \cite{hu2023suppression}. Thus, we arrange the proof into the following steps.

\noindent\textbf{Step 1. Study of Static Problem \eqref{eq:ksstatstokes}.} We first discuss the intuition of deriving the static problem \eqref{eq:ksstatstokes}. We start with dividing on both sides of \eqref{eq:ksstokesvel2} by $B$. This yields
$$
\frac{1}{B}(\p_t u + u\cdot\nabla u) - \Delta u + \frac{1}{B}\nabla p = g\rho (0,1)^T,\quad \divv u = 0.
$$
In the regime of large $B$, it is natural to treat the term $\frac{1}{B}(\p_t u + u\cdot\nabla u)$ perturbatively. In fact, it is reasonable to guess that the following velocity law dominates the velocity evolution:
{
\begin{equation}\label{eq:ksstatvel}
    -\Delta u + \nabla P = g\rho (0,1)^T,\; \divv u = 0,
\end{equation}
where we write $P = B^{-1}p$. We remark that the scaling $B^{-1}$ in front of the pressure is irrelevant due to incompressibility.} This motivates us to first understand the following system, which we will later refer to as the \textbf{static problem}:
\begin{subequations}
    \label{eq:ksstatstokes}
    \begin{equation}
        \p_t \rho + u\cdot \nabla \rho -\Delta \rho + \divv(\rho\nabla(-\Delta_N)^{-1}(\rho- \rho_m)) = 0,
    \end{equation}
    \begin{equation}
    \label{eq:ksstatvel2}
        -\Delta u + \nabla p = g\rho (0,1)^T,\quad \divv u = 0,
    \end{equation}
    \begin{equation}
    \label{eq:ksstatbc}
        \p_2 \rho\big|_{\p \Omega} = 0,\; u_2|_{\p\Omega} = 0,\; \p_2 u_1|_{\p\Omega} = 0,\; \rho(0,x) = \rho_0(x) \ge 0. 
    \end{equation}
\end{subequations}
We remark that we do not specify the initial datum for $u$. Indeed, the problem \eqref{eq:ksstatstokes} is closed by itself, due to a remarkable property of \eqref{eq:ksstatvel2}-\eqref{eq:ksstatbc} that one may reduce them to the following compact form:
\begin{equation}
    \label{eq:bs}
     u = -g\nabla^\perp(-\Delta_D)^{-2}\p_1\rho,
\end{equation}
thanks to Proposition \ref{prop:statstokes}. Here $-\Delta_D$ denotes homogeneous Dirichlet Laplacian, and $(-\Delta_D)^{-2} = (-\Delta_D)^{-1}\circ (-\Delta_D)^{-1}$. This renders no degree of freedom over choosing the initial datum for $u$ in problem \eqref{eq:ksstatstokes}. We will work with the more convenient form \eqref{eq:bs} instead of the original formulation. 

We also remark that the static problem \eqref{eq:ksstatstokes} is reminiscent of the Keller-Segel equation coupling to Darcy's law studied in \cite{hu2023suppression}, whose velocity law is given by:
\begin{equation}\label{darcy}
u = g\nabla^\perp (-\Delta_D)^{-1}\p_1 \rho.
\end{equation}
In \cite{hu2023suppression}, the authors exploited an anisotropic mixing effect that is sufficient to prevent chemotactic blowup \red{when $\rho$ achieves a large $L^2$ norm}. However, the velocity law \eqref{eq:bs} yields a more regular fluid flow. The extra regularity induces a weaker mixing effect {comparing to \eqref{darcy}}. To remedy this issue, we introduce a large $g$ so that the flow has large amplitude to generate sufficient mixing effect. \red{This treatment differs our work from \cite{hu2023suppression}.} In the large amplitute setting, we can achieve a crucial norm stabilizing property of \eqref{eq:ksstatstokes}. That is, the $L^2$ norm of $\rho$ stabilizes around a level $N_0$, whose size only depends on the initial datum $\rho_0$ in the regime of $g$ sufficiently large. The abovementioned property is formalized in the statement below:
\begin{prop}[Norm-Stabilizing Property]
\label{prop:barrier}
    Consider problem \eqref{eq:ksstatstokes} with nonnegative initial datum $\rho_0 \in H^1$. Assume that $t_0 \ge 0$ is inside the lifespan of the unique regular local solution $\rho(t,x)$. Then there exist $N_0 = N_0(\rho_m, \|\rho_0 - \rho_m\|_{L^2})$, a time $T_* = T_*(N_0, \rho_m)$, and $g_0 := g_0(N_0, \rho_m) > 0$, such that the following statement holds: if $\|\rho(t_0) - \rho_m\|_{L^2}^2 \le {\frac{N_0}{2}}$, then for all $g \ge g_0$, the regular solution $\rho(t,x)$ can be continued in interval $[t_0, t_0 + T_*]$ with estimate
    \begin{equation}
        \label{est:barrier1}
        \sup_{t \in [t_0, t_0 + T_*]}\|\rho(t) - \rho_m\|_{L^2}^2 \le {N_0}.
    \end{equation}
    Moreover, there exists \red{a time instance} $T \in \left[t_0 + \frac{T_*}{2}, t_0 + T_*\right]$ such that
    \begin{equation}
        \label{est:barrier2}
        \|\rho(T) - \rho_m\|_{L^2}^2 \le \frac{N_0}{8}.
    \end{equation}
\end{prop}
We will devote Section \ref{sec:stat} to the proof of this key result. As a direct consequence, we will also see that the static problem \eqref{eq:ksstatstokes} is globally regular.

\noindent\textbf{Step 2. Continuation Argument.} In fact, one may view Proposition \ref{prop:barrier} as a damping mechanism induced by the static problem \eqref{eq:ksstatstokes}. Keeping in mind that the full problem \eqref{eq:ksgen2} is a perturbation of \eqref{eq:ksstatstokes}, we employ a bootstrap argument to show that the solution to the full problem is $L^2$-close to the solution of the static problem for all times. Fixing initial condition $\rho_0$ and $u_0$, we choose large $N_0$ as in Proposition \ref{prop:barrier}. Define $\calT_0$ to be the largest time such that the following bootstrap assumption holds:
\begin{equation}
\label{bootstrap}
    \sup_{0 \le t \le \calT_0}\|\rho(t) - \rho_m\|_{L^2}^2 \le 9N_0.
\end{equation}
Our goal is to prove the following improved estimate:
\red{
\begin{prop}
    \label{prop:bootstrap}
    Consider initial data $\rho_0 \in H^1$ that is nonnegative and $u_0 \in V$. There exists a couple $(g_1,B_1) = (g_1(\rho_m, \|\rho_0 - \rho_m\|_{L^2}), B_1(\rho_m, \|\rho_0 - \rho_m\|_{L^2}, \|u_0\|_1))$ such that if $g \ge g_1$, $B \ge B_1g^2e^{g^2}$, the following improved estimate holds:
    \begin{equation}
        \label{est:improved}
        \sup_{0 \le t \le \calT_0}\|\rho(t) - \rho_m\|_{L^2}^2 \le 5N_0.
    \end{equation}
\end{prop}
}
\begin{comment}
\begin{rmk}
    The explicit choices of parameters $N_0, \mfg_0$ are listed in Remark \ref{rmk:statchoice}. The choices of $g_\delta, B_\delta$ for $\delta = 0,1$ are listed in Remark \ref{rmk:stokeschoice} and \ref{rmk:nschoice} respectively.
\end{rmk}
\end{comment}
The main ingredient of this step is based on various nonlinear estimates pertaining to Navier-Stokes equation equipped with {Lions boundary conditions}. We devote Section \ref{sec:stokes} to the proof of these estimates as well as that of Proposition \ref{prop:bootstrap}.

\noindent\textbf{Step 3. Closing the Proof of Theorem \ref{thm:stokeswp}.}
In the final step, we finish the proof of Theorem \ref{thm:stokeswp} using estimate \eqref{est:improved}:
\begin{proof}[Proof of Theorem \ref{thm:stokeswp}]
    Using Proposition \ref{prop:bootstrap} and regularity criteria in Theorem \ref{thm:criteria}, $[0,\calT_0]$ is both open and closed. This implies that $\calT_0 = \infty$ and we have obtained global existence of \eqref{eq:ksgen2}.
\end{proof}

\subsection{Organization of the Paper}
{The paper is organized as follows. In Section \ref{sec:prelim}, we introduce notations and conventions used throughout this paper. We will also lay out a functional-analytic framework for studying Stokes problem equipping Lions boundary condition. In Section \ref{sec:lwp}, we state the local well-posedness of the full system \eqref{eq:ksgen2} and the static problem \eqref{eq:ksstatstokes}; we also show corresponding $L^2$-based regularity criteria in the same section. In Section \ref{sec:stat}, we study the static problem \eqref{eq:ksstatstokes} and prove the key Proposition \ref{prop:barrier}, from which the global well-posedness of \eqref{eq:ksstatstokes} follows. In Section \ref{sec:stokes}, we perform nonlinear estimates and prove the improved estimate \eqref{est:improved}. That is, we prove Proposition \ref{prop:bootstrap}.\\}

%%%%%%%%%%%%%%%%%%%%%%%%%%%%%%%%%%%%%%%%%
\textbf{Acknowledgement.} The author acknowledges partial support from NSF-DMS grants 2006372, 2306726, and 2304392. The author would like to cordially thank Professor Alexander Kiselev for his kind guidance and support throughout the years. He thanks Professor Siming He for numerous insightful discussions, suggestions, and comments on an early draft. The author also thanks the hospitality of the Chinese University of Hong Kong. Finally, the author thanks anonymous referees for their constructive comments which greatly improve the presentation of this paper.

%%%%%%%%%%%%%%%%%%%%%%%%%%%%%%%%%%%

\section{Preliminaries}\label{sec:prelim}
\subsection{Notation and Conventions}
We will denote the usual $L^p$ norm by $\|\cdot\|_{L^p}$ or $\|\cdot\|_{L^p(\cdot)}$. The second notation will be adopted if we would like to emphasize the domain. Similarly, we will denote Sobolev $H^s$ norm by $\|\cdot\|_{s}$ for $s \ge 1.$ We define the space $H^1_0(\Omega)$ to be the completion of $C^\infty_c(\Omega)$ with respect to norm $\|\cdot\|_1$, and consider its dual space $ H^{-1}_0(\Omega)$. Moreover, the following equivalence is standard:
$$
\|f\|_{H^{-1}_0} = C\left(\int_\Omega f(-\Delta_D)^{-1}f dx\right)^{1/2},
$$
where $C$ is a constant that only depends on domain $\Omega$. We also recall that $\Delta_D$ denotes homogeneous Dirichlet Laplacian. We simply regard $C = 1$ as the exact value of this constant is inessential.

Recall the mean of $\rho$ on $\Omega:$ $\rho_m(t) = \frac{1}{|\Omega|} \int_\Omega \rho(t,x)\,dx.$ We remark that $\rho_m(t)$ is nonnegative and a conserved quantity along evolution given the initial datum $\rho_0$ being nonnegative, due to a parabolic maximum principle and the divergence structure of the Keller-Segel equation (see Proposition \ref{lem:mass} for the proof of this fact). Thus, we will omit the time dependence of the mean and write $\rho_m$. 

We also consider the following decomposition to any function $\rho \in L^2(\Omega)$ that would play a crucial role in our analysis of the static problem: we write $\rho = \brho + \trho$, where
 $$
        \brho(x_2) = \frac{1}{2\pi}\int_{\T}\rho(x_1,x_2)dx_1,\; \trho = \rho -\brho.
$$
Note that $\brho$ is exactly the projection onto the zeroth mode corresponding to direction $x_1$, and $\trho$ the orthogonal complement. We remark that $\brho - \rho_m$ and $\trho$ are orthogonal in $L^2(\Omega)$:
$$
\int_\Omega (\brho(x_2) - \rho_m)\trho(x_1, x_2)dx_1dx_2 = 0.
$$
One can connect the 1D and 2D $L^2$ norm for $\brho - \rho_m$ by a simple relation $\|\brho - \rho_m\|_{L^2(\Omega)}^2 = 2\pi\|\brho - \rho_m\|_{L^2([0,\pi])}^2$. In the rest of the paper, we will only use the 1D $L^2$ norm for $\brho$.
%For the sake of simplicity,
We always refer to the 1D $L^2$ norm when we write $\|\brho - \rho_m\|_{L^2}$.

We will denote a universal constant in the upper bounds by uppercase letters, and that in lower bounds by lowercase letters. Constants without subscripts, i.e. $c, C, \calC$, are subject to change from line to line. Constants with subscripts, i.e. $c_i, C_i, \calC_i$, are fixed once they are chosen. All such constants mentioned above only depend on domain $\Omega$ by default. We will also use the notation $C(X)$ to denote a constant $C$ depending on quantity $X$. For two positive quantities $X$, $Y$, we write $X \lesssim Y$ to mean that $X \le CY$ for some universal constant $C$ that might only depends on domain $\Omega$. $X \gtrsim Y$ is similarly defined. \red{We will use $I_i$, $i \in \N$, to denote terms that will be estimated; this notation will be local to each lemma and will be reused in the proofs of other lemmas.} Lastly, we also use summation convention. That is, the repeated indices are summed over.

\subsection{Steady Stokes Equation with Lions Boundary Condition}
\label{subsect:lions}
In this subsection, we introduce suitable functional-analytic settings for our study of Stokes equation equipped with Lions boundary condition. We shall first study the following static Stokes equation and investigate its well-posedness:
 \begin{equation}
    \label{eq:statstokes}
    \begin{cases}
    -\Delta u + \nabla p = f,& \text{ in }\Omega\\
    \divv u = 0,& \text{ in }\Omega\\
    u_2|_{\p\Omega} = 0,\; \p_2u_1|_{\p\Omega} = 0.
    \end{cases}
    \end{equation}
In particular, we will derive suitable elliptic estimates similar to those satisfied by the classical Stokes operator that corresponds to no-slip boundary condition for velocity (see \cite{constantin2020navier}).

We introduce the following function spaces following the spirit of \cite{kelliher2006navier,clopeau1998vanishing}:
$$
H = \{v \in L^2(\Omega) \times L^2(\Omega)\;:\; \divv v = 0,\; \int_\Omega v_1(x)dx = 0\},
$$
$$
V = \{v \in H^1(\Omega) \times H^1(\Omega) : \divv v = 0,\; v_2|_{\p\Omega} = 0,\; \int_\Omega v_1(x)dx = 0\},
$$
$$
W = \{v \in V \cap (H^2(\Omega))^2 : \p_2 v_1|_{\p\Omega} = 0\},
$$
where we equip $H$ with $L^2$ topology, $V$ with topology induced by the inner product $(u,v)_V = \int_\Omega \p_i u_j \p_i v_j dx$, and $W$ with $H^2$ topology. For $u \in V$, the following Poincar\'e inequality holds:
\begin{prop}
\label{prop:Poincare}
    Let $u \in V$. Then the following estimate holds:
    \red{
    \begin{equation}
    \label{est:Poincare}
    \|u\|_{L^2} \le C_P \|\nabla u\|_{L^2},
    \end{equation}
    where $C_P$ is a positive constant that only depends on the domain $\Omega$.}
\end{prop}
Moreover, we have the following well-posedness and regularity result concerning \eqref{eq:statstokes}:
\begin{prop}
    \label{prop:statstokes}
    Assuming $f \in H$, then \eqref{eq:statstokes} admits a unique solution $u \in W$ with estimate
    \begin{equation}
        \label{est:elliptic1}
        \|u\|_2 \le C\|f\|_{L^2}.
    \end{equation}
    In fact, $u$ assumes the following explicit formula:
    \begin{equation}
    \label{eq:bsorigin}
    u = \nabla^\perp(-\Delta_D)^{-2}(\p_2f_1 - \p_1f_2),
    \end{equation}
    where $\nabla^\perp := (-\p_2,\p_1)$. In addition, if $f \in H^s \cap H$, $s \ge 1$, we have the following improved regularity estimate:
    \begin{equation}
        \label{est:elliptic2}
        \|u\|_{s+2} \le C\|f\|_{s}.
    \end{equation}
\end{prop}

We delay the proofs of the Propositions above to the appendix.

%%%%%%%%%%%%%%%%%%%%%%%%%%%%%%%%%%%%%%%%%%%%%%%%%%%%%%%%%%%%%%%%%%%%
\section{Local Well-Posedness and Regularity Criteria}\label{sec:lwp}
In this section, we show the local well-posedness for regular solutions to problems \eqref{eq:ksgen2} and \eqref{eq:ksstatstokes}, and state $L^2$ regularity criteria for both problems. {We start with giving rigorous definitions of a strong solution and a regular solution to both the full system \eqref{eq:ksgen2} and the static problem \eqref{eq:ksstatstokes}.}

{
\begin{defn}
\begin{enumerate}
    \item Given initial data $\rho_0 \in H^1$, $u_0 \in V$, we say that the pair $(\rho(t,x), u(t,x))$ is a strong solution to \eqref{eq:ksgen2} on $[0,T]$, $T > 0$, if 
    \begin{align*}
    \rho \in C^0([0,T]; H^1) \cap L^2((0,T); H^2),&\; u \in C^0([0,T]; V) \cap L^2((0,T); W),\\
    \p_t \rho \in L^2([0,T];L^2),&\; \p_t u \in L^2([0,T]; H),
    \end{align*}
    and $(\rho, u)$ satisfies \eqref{eq:ksgen2} in the sense of distribution. Moreover, a solution is regular if it is strong and additionally
$$
\rho \in C^\infty((0,T] \times \Omega),\quad u \in C^\infty((0,T] \times \Omega).
$$
    \item Given initial datum $\rho_0 \in H^1$, we say that $\rho(t,x)$is a strong solution to \eqref{eq:ksstatstokes} on $[0,T]$, $T > 0$, if 
    \begin{align*}
    \rho \in C^0([0,T]; H^1) \cap L^2((0,T); H^2),\quad \p_t \rho \in L^2([0,T];L^2),
    \end{align*}
    and $\rho$ is said to be regular if it is strong and $\rho \in C^\infty((0,T] \times \Omega).$
\end{enumerate}
\end{defn}
}
With the definition above, we state the following local well-posedness result.

{
\begin{thm} \label{thm:lwpstokes}
Given initial datum $\rho_0 \in H^1$ (resp. $\rho_0 \in H^1, u_0 \in V$) such that $\rho_0 \ge 0$, there exists $T_{loc} = T_{loc}(\rho_0) > 0$ (resp. $\calT_{loc}= \calT_{loc}(\rho_0) > 0$) so that there exists a unique strong solution $\rho$ (resp. $(\rho, u)$) to \eqref{eq:ksstatstokes} (resp. \eqref{eq:ksgen2}) on $[0,T_{loc})$ (resp. $[0,\calT_{loc})$). Moreover, the strong solutions are regular.
\end{thm}
}
We conclude this section by stating the following regularity criteria for both \eqref{eq:ksgen2} and \eqref{eq:ksstatstokes}.

\begin{thm}
    \label{thm:criteria}
    If the maximal lifespan $T$ of the strong solution to \eqref{eq:ksgen2} (resp. \eqref{eq:ksstatstokes}) is finite, then necessarily
    \begin{equation}
        \label{eq:criteria}
        \lim_{t \to T}\int_0^t\|\rho(t) - \rho_m\|_{L^2}^2 dt = \infty.
    \end{equation}
\end{thm}

We briefly comment on the proof of Theorem \ref{thm:lwpstokes} and Theorem \ref{thm:criteria}, as the details follow almost verbatim to Section 2 of \cite{hu2023stokes}. To prove the existence part of Theorem \ref{thm:lwpstokes}, we use standard Galerkin's method to construct a solution to both \eqref{eq:ksgen2} and \eqref{eq:ksstatstokes} since $\Omega$ is a compact domain. Hence, it suffices to show $L^2$--based \textit{a priori} estimates as in \cite{hu2023stokes}. As for the \textit{a priori} estimates of \eqref{eq:ksgen2}, all estimates in Section 2 of \cite{hu2023stokes} carry over. The only extra estimate one has to make is on advection term $u\cdot\nabla u$, as we study the Navier-Stokes equation in our case. The treatment of such terms, however, are
standard in the study of well-posedness of 2D Navier-Stokes equations. We refer interested readers to the proof of \cite[Theorem V.2.10]{boyer2012mathematical} for related details. As for the \textit{a priori} estimates of \eqref{eq:ksstatstokes}, we repeat the same proof as we do for \eqref{eq:ksgen2}, with the only change of applying Proposition \ref{prop:statstokes} whenever we see $\|u\|_2$. Finally, uniqueness follows from the \textit{a priori} estimates that we have derived above. Higher regularity follows from standard parabolic smoothing. This ends the proof of Theorem \ref{thm:lwpstokes}.

The proof of Theorem \ref{thm:criteria} is the same as that of \cite[Theorem 1.2]{hu2023stokes} by studying the $L^2$ estimate of the $\rho$ equation. The reason why the proofs of \cite[Theorem 1.2]{hu2023stokes} can be carried over is that \cite[Theorem 1.2]{hu2023stokes} only use $\divv u = 0$, $\rho$ nonnegative, and $\|\rho\|_{L^1}$ non-increasing. These properties hold in both \eqref{eq:ksgen2} and \eqref{eq:ksstatstokes} (see Proposition \ref{lem:mass}).

%%%%%%%%%%%%%%%%%%%%%%%%%%%%%%%%%%%%%%%%%%%%%%%%%%%%%%5
\section{Analysis of the Static Problem}\label{sec:stat}
In this section, we perform a detailed analysis of \eqref{eq:ksstatstokes}, which serves as a starting point prior to the study of the full problem \eqref{eq:ksgen2}. In Section \ref{subsect:aprioristat}, we establish a series of \textit{a priori} estimates concerning both cell density $\rho$ and velocity field $u$ of the static problem. We will further show a Nash-type inequality that elucidates an anisotropic mixing effect in the regime of large $g$. Finally, we prove Proposition \ref{prop:barrier} in Section \ref{subsect:statgwp}.

\subsection{Key \textit{A Priori} Estimates}
\label{subsect:aprioristat}
As a first step, we establish several key energy estimates concerning the system \eqref{eq:ksstatstokes}. We first prove several \textit{a priori} estimates for the Keller-Segel equation advected by a general incompressible flow. Then we focus on showing estimates corresponding to the fluid equation \eqref{eq:ksstatvel2}, which turn out to be crucial in establishing the norm-stabilizing property and global well-posedness of \eqref{eq:ksstatstokes}.

\subsubsection{Estimates for a General Keller-Segel Equation with Advection}
In this section, we derive appropriate estimates for the Keller-Segel equation advected by a general incompressible flow $u$, which may or may not depend on the cell density $\rho$ itself:
\begin{equation}
    \label{eq:convectedks}
    \begin{split}
    \p_t \rho + u\cdot \nabla \rho -\Delta \rho + \divv(\rho\nabla(-\Delta_N)^{-1}(\rho- \rho_m)) = 0,\\
    \rho(0,x) = \rho_0(x) \ge 0,\; \p_2\rho|_{\p\Omega} = 0.
    \end{split}
\end{equation}
Here, $u$ is a smooth, divergence-free vector field that satisfies no-flux boundary condition (i.e. $u_2|_{\p\Omega} = 0$). We allow such generality in \eqref{eq:convectedks} so that all results in this subsection hold for both the static problem \eqref{eq:ksstatstokes} and the full problem \eqref{eq:ksgen2}. This is because an incompressible, viscous flow satisfying Lions boundary condition necessarily satisfies all requirements prescribed in \eqref{eq:convectedks}.

Our first result reveals that \eqref{eq:convectedks} conserves mass.
\begin{prop}
    \label{lem:mass}
    Assume that $\rho(t,x)$ is a regular solution to \eqref{eq:convectedks} on $[0,T]$ \red{with $u$ being $C^1$ and divergence-free.} Then $\rho(t,x) \ge 0$ and $\|\rho(t,\cdot)\|_{L^1} = \|\rho_0\|_{L^1}$ for all $t \in [0,T]$.
\end{prop}
\begin{proof}
    First, $\rho(t,x) \ge 0$ for all $t \in [0,T]$ comes from the facts that $\rho_0 \ge 0$ and parabolic maximum principle. To prove the conservation of mass, we use nonnegativity of $\rho(t,x)$ to obtain that
    \begin{align*}
        \frac{d}{dt}\|\rho(t,\cdot)\|_{L^1} &= \frac{d}{dt}\int_\Omega \rho(t,x)dx = \int_\Omega (-u\cdot \nabla \rho + \Delta \rho - \divv(\rho\nabla(-\Delta_N)^{-1}(\rho- \rho_m)))(t,x) dx = 0,
    \end{align*}
    where we integrated by parts and used that $u$ is divergence-free, $u_2|_{\p\Omega} = 0$, and $\p_2 \rho|_{\p\Omega} = 0$.
\end{proof}

The next result, originally proved in \cite{hu2023suppression}, gives two variants of $L^2$ energy estimates of \eqref{eq:convectedks}, which play a major role in establishing the norm-stabilizing property of \eqref{eq:ksstatstokes}. {We point out that \cite[Proposition 4.2 and Corollary 4.3]{hu2023suppression} are proved where $u$ is governed by Darcy's law \eqref{darcy}. However, their proofs only used that $u$ is divergence-free and satisfies no-flux boundary condition. Hence, the proof of the result below follows verbatim from that of \cite[Proposition 4.2 and Corollary 4.3]{hu2023suppression}.}

\begin{prop}[Essentially Proposition 4.2, Corollary 4.3 of \cite{hu2023suppression}]
\label{lem:naive}
    Assume $\rho(t,x)$ to be a regular, nonnegative solution to \eqref{eq:convectedks} on $[t_0,t_0 + T]$ for some $t_0 \ge 0$ and $T > 0$. Then for any $t \in (t_0, t_0 + T)$, there exists a constant $C_0$ depending only on the domain $\Omega$ such that
    \begin{equation}
        \label{est:naive0}
        \frac{d}{dt}\|\rho - \rho_m\|_{L^2}^2 + \|\nabla \rho\|_{L^2}^2 \le C_0\|\rho - \rho_m\|_{L^2}^4 + 2\rho_m \|\rho - \rho_m\|_{L^2}^2.
    \end{equation}
    \red{Furthermore, if $\|\rho(t,\cdot) - \rho_m\|_{L^2}^2 \geq \rho_m/4$ for all $t \in [t_0, t_0 + T]$, then there exists a positive constant $C_1$ depending only on the domain $\Omega$ such that
    \begin{align}
    \label{est:naive}
        \frac{d}{dt}\|\rho - \rho_m\|_{L^2}^2 + \|\nabla \rho\|_{L^2}^2 &\le C_1\|\rho - \rho_m\|_{L^2}^4,
    \end{align}
    and
    \begin{align}
    \label{est:mainenergy}
    \frac{d}{dt}\|\rho - \rho_m\|_{L^2}^2 &\le -\frac{1}{C_1\rho_m^4}\|\brho - \rho_m\|_{L^2}^6+ C_1 \|\trho\|_{L^2}^4,
    \end{align}
    for any $t \in [t_0,t_0 + T]$.}
\end{prop}

%\begin{rmk}
%    The constants $C_0, C_1$ only depend on the domain $\Omega$ as they either come from H\"older inequalities or Gagliardo-Nirenberg-Sobolev inequalities.
%\end{rmk}

The next proposition shows that one can upgrade the $L^\infty_tL^2_x$ control of $\rho$ to an $L^\infty_tL^\infty_x$ control:
\begin{prop}
\label{lem:L2toLinfty}
Let $\rho(t,x)$ be a regular solution to \eqref{eq:convectedks} on $[t_0,t_0 + T]$ for some $t_0\ge 0$ \red{and $N \ge \max(1, \rho_m)$ be an arbitrary number.} If $\|\rho(t,\cdot) - \rho_m\|_{L^2} \le 2N$ for all $t \in [t_0,t_0 + T]$, then we also have $\sup_{t_0 \le t \le t_0 + T}\|\rho(t,\cdot) - \rho_m\|_{L^\infty} \le CN^2$, where $C$ is a domain dependent constant.
\end{prop}
\begin{proof}
    The proof follows verbatim from that of \cite[Proposition 9.1]{kiselev2016suppression}, where it is handled for $\T^2$. We can do this since all integrations by parts do not introduce any boundary terms due to the Neumann boundary condition that $\rho$ satisfies as well as the no-flux boundary condition satisfied by $u$.
\end{proof}

Finally, we establish a corollary shows a doubling time for the $L^2$ norm $\|\rho - \rho_m\|_{L^2}^2$.
\begin{cor}
    \label{cor:chartime}
    Suppose $\rho(t,x)$ is a regular solution to \eqref{eq:convectedks} on $[t_0, t_0 + T]$ for some $t_0 \ge 0$, \red{$T > 0$, and $N > 0$ being an arbitrary number.} If $\|\rho(t_0) - \rho_m\|_{L^2}^2 \le N$, then there exists a time \red{$T_*(N,\rho_m) = \frac{\min(N^{-1},\rho_m^{-1})}{8(1+C_0)} =: c_0\min(N^{-1},\rho_m^{-1})$} such that
    \begin{equation}
    \label{est:naiveL21}
        \sup_{t_0\le t \le t_0 + T_*(N,\rho_m)}\|\rho(t) - \rho_m\|_{L^2}^2 \le 2N,
    \end{equation}
    and
    \begin{equation}
    \label{est:naiveL2}
    \int_{t_0}^{t_0 + T_*(N,\rho_m)}\|\nabla \rho(t)\|_{L^2}^2 dt \le 2N.
    \end{equation}
    \red{Here, $C_0$ is fixed as in Proposition \ref{lem:naive}.}
\end{cor}
\begin{proof}
    To show \eqref{est:naiveL21}, we employ a standard barrier argument: suppose that there exists a time $\tau \in [t_0, t_0 + T_*(N,\rho_m)]$ such that $\|\rho(\tau) - \rho_m\|_{L^2}^2 = 2N$, and $\|\rho(t) - \rho_m\|_{L^2}^2 < 2N$ for any $t \in [t_0,\tau)$. Integrating \eqref{est:naive0} from $t_0$ to $\tau$, we have
    {
    \begin{align*}
        2N &= \|\rho(\tau) - \rho_m\|_{L^2}^2 \le N + 2\rho_m\int_{t_0}^\tau \|\rho(s) - \rho_m\|_{L^2}^2 dx + C_0\int_{t_0}^{\tau}\|\rho(s) - \rho_m\|_{L^2}^4 ds\\
        &\le N + 4\rho_m NT_* + 4N^2C_0T_* < 2N,
    \end{align*}
    where we used the definition of $T_*$ in the final inequality. However, this yields a contradiction.}
    
    To show \eqref{est:naiveL2}, we integrate \eqref{est:naive0} from $t_0$ to $t_0 + T_*(N,\rho_m)$, which yields:
    {
    \begin{align*}
        \|\rho(t_0 + T_*(N,\rho_m)) &- \rho_m\|_{L^2}^2 + \int_{t_0}^{t_0 + T_*(N,\rho_m)}\|\nabla \rho(t)\|_{L^2}^2 dt\\
        &\le N + 2\rho_m\int_{t_0}^{t_0 +T_*(N,\rho_m)}\|\rho(t) - \rho_m\|_{L^2}^2 dt + C_0\int_{t_0}^{t_0 + T_*(N,\rho_m)}\|\rho(t) - \rho_0\|_{L^2}^{4}dt\\
        &\le N + 2\rho_m(2N)T_* + C_0(4N^2)T_* < N + \frac{N}{2(1+C_0)} + \frac{C_0N}{2(1+C_0)} < 2N.
    \end{align*}
    }
    The proof is thus complete.
\end{proof}
\begin{comment}
\begin{rmk}
For the rest of this work, $T_*(N,\rho_m)$ is defined as in Corollary \ref{cor:chartime} after one fixes $M$ and $\rho_m$.
\end{rmk}
\end{comment}

\red{\subsubsection{Estimates for Static Problem \eqref{eq:ksstatstokes}}
Now, we turn to some estimates concerning the system \eqref{eq:ksstatstokes}, which crucially exploit the structure of the Biot-Savart-type relation \eqref{eq:bs}. As our first step, we prove the following key estimate which displays the smallness of a mixing-type norm in the regime of large parameter $g$.
}
\begin{prop}
\label{lem:apriori}
Let $\rho(t,x)$ be a regular solution to equation \eqref{eq:ksstatstokes} on $[t_0, t_0 + T]$ for some $t_0 \ge 0$ and $T > 0$. \red{Let $N > 0$ be an arbitrary number. If $\|\rho(t_0) - \rho_m\|_{L^2}^2 \le N$}, the following bound holds:
\begin{equation}
\label{est:H-2bd}
\int_{t_0}^{t_0 + T_*(N,\rho_m)}\|(-\Delta_D)^{-1}\p_1\rho\|_{L^2}^2 dt \le C(\rho_m)g^{-1}.
\end{equation}
\end{prop}
\begin{proof}
    On one hand, we compute that
    \begin{align}
    \label{apriori:aux1}
    \int_{t_0}^{t_0 + T_*(N,\rho_m)} \|(-\Delta_D)^{-1}\p_1\rho\|_{L^2}^2 dt &= \int_{t_0}^{t_0 + T_*(N,\rho_m)} \int_\Omega (-\Delta_D)^{-1}\p_1\rho (-\Delta_D)^{-1}\p_1 \rho dxdt\notag\\
    &= \int_{t_0}^{t_0 + T_*(N,\rho_m)}\int_\Omega \p_1\rho (-\Delta_D)^{-2}\p_1 \rho dxdt\notag\\
    &= -\int_{t_0}^{t_0 + T_*(N,\rho_m)}\int_\Omega \rho \p_1(-\Delta_D)^{-2}\p_1 \rho dxdt\notag\\
    &= \frac{1}{g}\int_{t_0}^{t_0 + T_*(N,\rho_m)}\int_\Omega \rho u_2 dxdt.
    \end{align}
    Note that we have used the self-adjointness of $(-\Delta_D)^{-1}$ in the second equality, integration by parts in the third equality, and \eqref{eq:bs} in the last equality. On the other hand, we test the $\rho$-equation of \eqref{eq:ksstatstokes} by $x_2$, which yields:
    \begin{equation}\label{est:mixaux1}
    \begin{split}
        \int_\Omega \rho u_2 dx &= -\int_\Omega x_2(u\cdot\nabla \rho) dx = \int_\Omega x_2\left(\p_t \rho - \Delta \rho + \divv(\rho\nabla(-\Delta_N)^{-1}(\rho - \rho_m))\right)dx =: \sum_{i = 1}^3 I_i,
        \end{split}
    \end{equation}
    where the first equality follows from integration by parts and no-flux boundary condition satisfied by $u$. Integrating in time from $t_0$ to $t_0 + T_*(N,\rho_m)$ and using Proposition \ref{lem:mass} as well as the bound \eqref{est:naiveL2}, we estimate time integrals of $I_1, I_2$ by
    \begin{align*}
        \left|\int_{t_0}^{t_0 + T_*(N,\rho_m)}I_1dt\right| &= \left|\int_\Omega x_2(\rho(t_0 + T_*(N,\rho_m)) - \rho(t_0))dx\right| \le C\|\rho_0\|_{L^1} {\le C\rho_m},\\
        \left|\int_{t_0}^{t_0 + T_*(N,\rho_m)}I_2dt\right| &= \left|\int_{t_0}^{t_0 + T_*(N,\rho_m)}\int_\Omega \p_2 \rho dxdt\right| \le CT_*(N,\rho_m)^{1/2}\left(\int_{t_0}^{t_0 + T_*(N,\rho_m)}\|\nabla\rho(t)\|_{L^2}^2dt\right)^{1/2}\le C.
    \end{align*}
    Here, we used the definition of $T_*(N,\rho_m)$ and \eqref{est:naiveL2} in the final inequality of $I_2$ estimate. To estimate the term $I_3$ in \eqref{est:mixaux1}, we use Sobolev embedding, elliptic estimates, and interpolation inequality to obtain that
    \begin{align*}
            I_3 &= \int_\Omega \rho \p_2(-\Delta_N)^{-1}(\rho - \rho_m) \le \|\rho\|_{L^2}\|\p_2(-\Delta_N)^{-1}(\rho - \rho_m)\|_{L^2}\\
            &\le (\|\rho - \rho_m\|_{L^2} + {\sqrt{2\pi}}\rho_m)\|\p_2(-\Delta_N)^{-1}(\rho - \rho_m)\|_{W^{1,6/5}}\\
            &\le C(\|\rho - \rho_m\|_{L^2} + {\sqrt{2\pi}}\rho_m)\|\rho - \rho_m\|_{L^{6/5}}\\
            &\le C\rho_m^{2/3}\|\rho - \rho_m\|_{L^2}^{1/3} (\|\rho - \rho_m\|_{L^2} + \rho_m)\\
            &= {C(\rho_m^{5/3}N^{1/6} + \rho_m^{2/3}N^{2/3}).}
        \end{align*}
    Then we may estimate the time integral of $A_3$ by
    {
    \begin{align*}
    \int_{t_0}^{t_0 + T_*(N,\rho_m)}I_3dt &\le C\min(N^{-1},\rho_m^{-1})(\rho_m^{5/3}N^{1/6} + \rho_m^{2/3}N^{2/3}) \\
    &\le \begin{cases}
        \rho_m^{2/3} + \rho_m^{-1/3},& N \le 1,\\
        \rho_m^{5/3} N^{-5/6} + \rho_m^{2/3}N^{-1/3},& N > 1
    \end{cases} \le C(\rho_m).
    \end{align*}
    }
    Collecting the estimates above yields:
    \begin{equation}
    \left|\int_{t_0}^{t_0 + T_*(N,\rho_m)}\int_\Omega \rho u_2 dxdt \right|\le C(\rho_m),
    \end{equation}
    Therefore by \eqref{apriori:aux1}:
    $$
    \int_{t_0}^{t_0 + T_*(N,\rho_m)} \|(-\Delta_D)^{-1}\p_1\rho\|_{L^2}^2 dt \le C(\rho_m)g^{-1},
    $$
    which concludes the proof.
\end{proof}

\red{Note that the norm $\|(-\Delta_D)^{-1}\p_1 \rho\|_{L^2}$ could be viewed as a quantity that encodes information on mixing. This is because, formally, this quantity has the same scaling as the classical mix-norm $\|\rho - \rho_m\|_{H^{-1}_0}$ (see \cite{iyer2014lower} for example). Since generally smallness of $\|\rho - \rho_m\|_{H^{-1}}$ induces enhanced dissipation by creating large gradients, it is tempting to uncover a similar effect by exploiting the smallness of $\|(-\Delta_D)^{-1}\p_1 \rho\|_{L^2}$ in our setting. In order to achieve this goal, we follow the approach taken in \cite{hu2023suppression}, which studies a similar quantity $\|\p_1 \rho\|_{H^{-1}_0}$.} The technical ingredient that we need is a Nash's inequality with improved coefficient: let us recall the following result from \cite{hu2023suppression}, which is valid for any function $\rho \in L^2([s,r], H^1(\Omega))$, $0 \le s < r \le \infty$:

\begin{prop}[Corollary 6.2 of \cite{hu2023suppression}]
\label{prop:concentration}
Assume that for some \red{$K > 0$}, there is
    \begin{equation}\label{Hassump}
 %       \int_{s}^{r}\|\trho\|_{L^2}^2 dt \ge 2^N(r - s),\; \int_{s}^{r} \|\p_{x_1} \rho\|_{H^{-1}_{0}}^2 dt < C_2 N^{-1}2^N(r - s),
        \int_{s}^{r} \|\p_1 \rho\|_{H^{-1}_{0}}^2 dt \leq K^{-1} \int_{s}^{r}\|\trho\|_{L^2}^2 dt.
    \end{equation}
%    for some sufficiently large $N \geq N_0.$
Then there exists a constant $C > 0$ depending only on $\Omega$ such that
        \begin{equation}\label{GNN1}
        \int_{s}^{r}\|\trho\|_{L^2}^2 \,dt \le CK^{-1/4}\left(\int_{s}^{r}\|\trho\|_{L^1}^2 \,dt\right)^{1/2} \left(\int_{s}^{r}\|\nabla \trho\|_{L^2}^2 \, dt\right)^{1/2}.
        \end{equation}
\end{prop}
\red{Note that the original statement in Corollary 6.2 of \cite{hu2023suppression} only considers the case when $K \ge 1$. However, one could extend the result for any $K > 0$ in a rather trivial way: for $K \in (0,1)$, we simply invoke the original Nash's inequality for $\trho$ without using the assumption \eqref{Hassump}. The fact that $K^{-1/4} > 1$ immediately yields the result.}

\red{With the proposition above, we may conclude the following crucial lemma, which precisely characterizes the creation of large gradients when the norm $\|(-\Delta_D)^{-1}\p_1 \rho\|_{L^2}$ is small in an appropriate sense.
\begin{lem}
\label{lem:improvedNash}
    Fix $g > 0$. Assume that there is some constant $C > 0$ so that 
    \begin{equation}
    \label{aux1}
    \int_s^r \|(-\Delta_D)^{-1}\p_1 \rho\|_{L^2}^2 dt \le Cg^{-9/10}\left(\int_s^r \|\trho\|_{L^2}^2 dt\right)^2.
    \end{equation}
    Then the following estimate holds:
    \begin{equation}
        \label{aux2}
        \int_s^r \|\nabla\rho\|_{L^2}^2 dt \ge \tilde{C}g^{9/50}\left(\int_s^r \|\rho - \rho_m\|_{L^1}^2 dt\right)^{-4/5}\left(\int_s^r \|\trho\|_{L^2}^2 dt \right)^{8/5},
    \end{equation}
    for some constant $\tilde{C} > 0$ that only depends on the domain $\Omega$ and constant $C$ appearing in \eqref{aux1}.
\end{lem}
}
\begin{proof}
    Using definition of $H^{-1}_0$ norm and Cauchy-Schwarz inequality, we have
    $$
    \|\p_1\trho\|_{H^{-1}_0}^2 = \int_\Omega \p_1\trho (-\Delta_D)^{-1}\p_1 \trho dx \le \|\p_1\trho\|_{L^2}\|(-\Delta_D)^{-1}\p_1 \rho\|_{L^2} \le \|\nabla \rho\|_{L^2}\|(-\Delta_D)^{-1}\p_1 \rho\|_{L^2},
    $$
    where we used the fact that $\p_1 \rho = \p_1 \trho$. Integrating the above inequality from $s$ to $r$ and applying Cauchy-Schwarz inequality again, we have
    \begin{align*}
    \int_s^r \|\p_1\trho\|_{H^{-1}_0}^2 dt &\le \left(\int_s^r\|(-\Delta_D)^{-1}\p_1 \rho\|_{L^2}^2dt\right)^{1/2} \left(\int_s^r\|\nabla \rho\|_{L^2}^2 dt\right)^{1/2}\\
    &\le C^{1/2}g^{-9/20}\left(\int_s^r\|\nabla \rho\|_{L^2}^2 dt\right)^{1/2}\int_s^r \|\trho\|_{L^2}^2 dt.
    \end{align*}
    Then an application of Proposition \ref{prop:concentration} via replacing $N$ by $C^{-1/2}g^{9/20}\left(\int_s^r\|\nabla \rho\|_{L^2}^2 dt\right)^{-1/2} > 0$ implies \eqref{aux2} after some rearrangements of the inequality.
\end{proof}
\subsection{Norm-Stabilizing Property of the Static Problem}
\label{subsect:statgwp}
With the \textit{a priori} estimates set up in previous sections, we are ready to prove the norm-stabilizing property (i.e. Proposition \ref{prop:barrier}), from which the global existence of the static problem \eqref{eq:ksstatstokes} follows.

\begin{proof}[Proof of Proposition \ref{prop:barrier}]
    \red{We first fix our choice of $N_0$ and $T_*$ as follows: choose $N_0$ so that
    \begin{align}
        N_0 \ge \max\{1, 2\rho_m, 2\|\rho_0 - \rho_m\|_{L^2}^2\},\label{N0choice1}\\ 
        N_0 \ge \frac{32^3C_1\rho_m^4(15/8 + 2C_1c_0)}{4c_0}\label{N0choice2},
    \end{align}
    and set $T_* = c_0\min(2N_0^{-1}, \rho_m) = 2c_0N_0^{-1}$. Recall that $C_0, C_1$ are universal constants fixed in Proposition \ref{lem:naive}; $c_0$ is also a universal constant fixed in Corollary \ref{cor:chartime}. Moreover, we fix $g \ge g_0(N_0, \rho_m)$, where the threshold $g_0$ will be chosen sufficiently large in the course of the proof. }
    
    We first observe that by Corollary \ref{cor:chartime}, given $\|\rho(t_0) - \rho_m\|_{L^2}^2 = N_0/2$, we have
    \begin{equation}
        \label{est:gradctrl}
        \int_{t_0}^{t_0 + T_*}\|\nabla \rho\|_{L^2}^2 dt \le N_0,
    \end{equation}
    and in particular \eqref{est:barrier1} holds.

    We thus focus on the proof of \eqref{est:barrier2}. Assume for the sake of contradiction that $\|\rho(t) - \rho_m\|_{L^2}^2 > N_0/8$ for any $t \in \left[t_0 + \frac{T_*}{2}, t_0 + T_*\right]$. For $g \ge g_0$, we discuss the following dichotomy:\\
    \textbf{Case 1: Small $\|\trho\|_{L^2}$.}
    \begin{equation}\label{aux3}\int_{t_0 + T_*/2}^{t_0 + T_*}\|\trho\|_{L^2}^2 dt \le c_0g^{-1/20}.\end{equation}
    In this case, we first observe that due to the contradiction assumption, there is a lower bound:
        $$
        \int_{t_0 + T_*/2}^{t_0 + T_*}\|\rho - \rho_m\|_{L^2}^2 dt \ge \frac{T_*N_0}{16} = \frac{c_0}{16}.
        $$
        Combining with \eqref{aux3} and using orthogonality between $\brho$ and $\trho$, we have
        \begin{equation}
            \label{aux4}
            \int_{t_0 + T_*/2}^{t_0 + T_*}\|\brho - \rho_m\|_{L^2}^2 dt \ge \left(\frac{1}{16} - g^{-1/20}\right)c_0 \ge \frac{c_0}{32},
        \end{equation}
        {if we choose
        \begin{equation}
            \label{g0cond1}
            g_0 \ge 32^{20}.
        \end{equation}
        }
        On the other hand, an application of H\"older's inequality implies that
        \begin{align*}
            \int_{t_0 + T_*/2}^{t_0 + T_*}\|\brho - \rho_m\|_{L^2}^2 dt &\le \left(\frac{T_*}{2}\right)^{2/3}\left(\int_{t_0 + T_*/2}^{t_0 + T_*}\|\brho - \rho_m\|_{L^2}^6 dt\right)^{1/3}.
        \end{align*}
    Combining the inequality above with \eqref{aux4}, we have
    \begin{equation}
    \label{aux5}
    \int_{t_0 + T_*/2}^{t_0 + T_*}\|\brho - \rho_m\|_{L^2}^6 dt \ge \frac{4}{T_*^2}\left(\int_{t_0 + T_*/2}^{t_0 + T_*}\|\brho - \rho_m\|_{L^2}^2 dt\right)^3 \ge \frac{4c_0^3}{32^3T_*^2} = \frac{4c_0N_0^2}{32^3}.
    \end{equation}
    \red{Now, we shall invoke energy estimate \eqref{est:mainenergy} and integrate the differential inequality from $T_*/2$ to $T_*$. Note that we can do this since $\|\rho(t) - \rho_m\|_{L^2}^2 > N_0/8 \ge \rho_m/4$ for any $t \in \left[t_0 + \frac{T_*}{2}, t_0 + T_*\right]$ by our contradiction assumption, and our choice of $N_0$ in \eqref{N0choice1}. Therefore, we arrive at
    \begin{align*}
        \|\rho(t_0 + T_*) - \rho_m\|_{L^2}^2 &\le \|\rho(t_0 + T_*/2) - \rho_m\|_{L^2}^2 - \frac{1}{C_1\rho_m^4}\int_{t_0 + T_*/2}^{t_0 + T_*} \|\brho - \rho_m\|_{L^2}^6 dt + C_1 \int_{t_0 + T_*/2}^{t_0 + T_*} \|\trho\|_{L^2}^4 dt\\
        &\le 2N_0 - \frac{4c_0}{32^3C_1\rho_m^4}N_0^2 + C_1\sup_{t_0 + T_*/2 \le t \le t_0 + T_*/2}\|\rho - \rho_m\|_{L^2}^2\int_{t_0 + T_*/2}^{t_0 + T_*} \|\trho\|_{L^2}^2 dt\\
        &\le 2N_0 - \frac{4c_0}{32^3C_1\rho_m^4}N_0^2 + 2C_1c_0N_0g^{-1/20} \le 2N_0 - \frac{4c_0}{32^3C_1\rho_m^4}N_0^2 + 2C_1c_0N_0.
    \end{align*}
    }
    Here, we used \eqref{aux5} in the second inequality, \eqref{aux3} in the third inequality, and $g \ge g_0 > 1$ in the last inequality.
    \red{Then by \eqref{N0choice2}, a straightforward computation yields}
    $$
     \|\rho(t_0 + T_*) - \rho_m\|_{L^2}^2 \le 2N_0 - \frac{4c_0}{32^3C_1\rho_m^4}N_0^2 + 2C_1c_0N_0 < \frac{N_0}{8},
    $$
    which is a contradiction to our assumption.
    
    \red{\noindent\textbf{Case 2: Large $\|\trho\|_{L^2}$.}
    \begin{equation}
        \int_{t_0 + T_*/2}^{t_0 + T_*}\|\trho\|_{L^2}^2 dt > c_0g^{-1/20}. \label{aux44}
    \end{equation}
    In this case, we first use estimate \eqref{est:H-2bd} to obtain that
    $$
    \int_{t_0 + T_*/2}^{t_0 + T_*} \|(-\Delta_D)^{-1}\p_1\rho\|_{L^2}^2 dt \le C(\rho_m)g^{-1} < C(\rho_m)c_0^{-2}g^{-9/10}\left(\int_{t_0 + T_*/2}^{t_0 + T_*}\|\trho\|_{L^2}^2 dt\right)^2.
    $$
    Then, we can invoke Lemma \ref{lem:improvedNash} with $C = C(\rho_m)c_0^{-2}$ to conclude that
    \begin{equation}
        \label{est:largediff1}
        \begin{split}
        \int_{t_0 + T_*/2}^{t_0 + T_*} \|\nabla\rho\|_{L^2}^2 dt &\ge C(\rho_m)g^{9/50}\left(\int_{t_0 + T_*/2}^{t_0 + T_*} \|\trho\|_{L^1}^2 dt\right)^{-4/5}\left(\int_{t_0 + T_*/2}^{t_0 + T_*} \|\trho\|_{L^2}^2 dt \right)^{8/5}\\
        &\ge C(\rho_m)T_*^{-4/5}g^{9/50}g^{-2/25} \ge C(\rho_m)N_0^{4/5}g^{1/10}.
        \end{split}
    \end{equation}
    Here, we applied Proposition \ref{lem:mass}, \eqref{aux44} in the second inequality, and used the definition of $T_*$ in the final inequality. Now choose $g_0$ sufficiently large that
    \begin{equation}
        \label{g0cond3}
        C(\rho_m)N_0^{4/5}g^{1/10} > N_0.
    \end{equation}
    Then together with \eqref{est:largediff1}, we would arrive at $\int_{t_0 + T_*/2}^{t_0 + T_*} \|\nabla\rho\|_{L^2}^2 dt > N_0$. However, this contradicts the \textit{a priori} estimate \eqref{est:gradctrl}.
    }
Up until this point, both cases in the dichotomy lead to a contradiction and we have completed the proof.
\end{proof}

\begin{rmk}
\label{rmk:statchoice}
\begin{enumerate}
    \item \red{The specific choice of parameters $N_0, T_*,$ and $g_0$ are summarized as follows: we first fix $N_0$ according to \eqref{N0choice1}, \eqref{N0choice2}, and we set $T_* = 2c_0N_0^{-1}$. After $N_0, T_*$ are fixed, we choose $g_0$ sufficiently large that \eqref{g0cond1} and \eqref{g0cond3} hold.}

    \item {We also remark that the choice of $N_0 \ge 2\|\rho_0 - \rho_m\|_{L^2}^2$ does not play a role in the proof of Proposition \ref{prop:barrier}. However, it is a rather technical choice that will play a role in the proofs of Corollary \ref{cor:statglobal} and Proposition \ref{prop:bootstrap}, which rely on a barrier argument.}
\end{enumerate}
\end{rmk}

Note that we can conclude global well-posedness of system \eqref{eq:statstokes} from Proposition \ref{prop:barrier}.
\begin{cor}\label{cor:statglobal}
    For arbitrary nonnegative initial data $\rho_0 \in H^1$, there exists $
    g_0 = g_0(\rho_m) > 0$ sufficiently large that the solution of the static problem \eqref{eq:ksstatstokes} is globally regular for all $g \ge g_0$.
\end{cor}

\begin{proof}
    Let $g_0, N_0$ be chosen as in Proposition \ref{prop:barrier} (and therefore $T_*$ is fixed). Let $t_0 \ge 0$ be the first time such that $\|\rho - \rho_m\|_{L^2}^2 = N_0$. If $t_0 = \infty$, then this means that $\|\rho(t) - \rho_m\|_{L^2}^2 \le N_0$ for all $t \ge 0$ as $N_0 \ge \|\rho_0 - \rho_m\|_{L^2}^2$. Hence by the regularity criterion stated in Theorem \ref{thm:criteria}, we are already done. We may then suppose that $t_0 < \infty$. By Proposition \ref{prop:barrier}, there exists $T_1 \in [t_0 + T_*/2, t_0 + T_*]$ such that $\|\rho(t_0 + T_1) - \rho_m\|_{L^2}^2 = N_0/8 < N_0$. Moreover, we know that $\sup_{0 \le t \le t_0 + T_*}\|\rho(t) - \rho_m\|_{L^2}^2 \le 2N_0$. Then let $t_1$ to be the first time after $t_0 + T_1$ such that $\|\rho - \rho_m\|_{L^2}^2 = N_0$. Repeating the argument above, there exists $T_2 \in [t_1 + T_*/2, t_1 + T_*]$ such that $\|\rho(t_1 + T_2) - \rho_m\|_{L^2}^2 = N_0/8$ and $\sup_{t_0 + T_1 \le t \le t_1 + T_*}\|\rho(t) - \rho_m\|_{L^2}^2 \le 2N_0$. Repeating such steps indefinitely and using $T_* > 0$:
    $$
    \sup_{t \ge 0}\|\rho(t) - \rho_m\|_{L^2}^2 \le 2N_0 < \infty,
    $$
    which implies global well-posedness by the $L^2$ criterion in Theorem \ref{thm:criteria}.
\end{proof}

%%%%%%%%%%%%%%%%%%%%%%%%%%%%%%%
\section{Analysis of the Full Keller-Segel-Navier-Stokes Problem}\label{sec:stokes}
\label{sect:stokes}
In this section, we study the full Keller-Segel-Navier-Stokes system \eqref{eq:ksgen2}, and we prove Proposition \ref{prop:bootstrap}. For the readers' convenience, let us recall the system \eqref{eq:ksgen2} here:
\begin{subequations}
    \label{eq:ksstokes}
    \begin{equation}
    \label{eq:ksstokesden}
        \p_t \rho + u\cdot \nabla \rho -\Delta \rho + \divv(\rho\nabla(-\Delta_N)^{-1}(\rho- \rho_m)) = 0,
    \end{equation}
    \begin{equation}
    \label{eq:ksstokesvel}
        \p_t u + u\cdot \nabla u - B\Delta u + \nabla p = Bg\rho (0,1)^T,\quad \divv u = 0,
    \end{equation}
    \begin{equation}
        \p_2 \rho\big|_{\p \Omega} = 0,\; u_2|_{\p \Omega} = 0,\; \p_2u_1|_{\p \Omega} = 0,
    \end{equation}
    \begin{equation}
    \rho(0,x) = \rho_0(x) \ge 0,\; u(0,x) = u_0(x),
    \end{equation}
\end{subequations}

As pointed out in Section \ref{subsect:strat}, we will prove Proposition \ref{prop:bootstrap} with bootstrap assumption \eqref{bootstrap}. {We will do so by appropriately setting up a comparison scheme between the full problem \eqref{eq:ksstokes} and the static problem \eqref{eq:ksstatstokes}. By examining Proposition \ref{prop:barrier} more closely, we observe that the damping effect induced by the static problem \eqref{eq:ksstatstokes} becomes significant only when the $L^2$ norm of the cell density becomes sufficiently large i.e. $N_0$ as specified in the statement of Proposition \ref{prop:barrier}. Therefore, in the case of the full problem \eqref{eq:ksstokes}, we decompose the full solution $(\rho,u)$ into the static problem dynamics $(\rho_s, u_s)$ and the remainder term $(r,v)$ as $L^2$ norm of $\rho$ turns to the level of $N_0$. We then show suitable estimates for remainder term $(r,v)$, and show that they can be dominated by the damping effect introduced by the static part.} 

\red{From now, given initial data $(\rho_0, u_0)$, we fix $N_0$ and $T_*$ as in Proposition \ref{prop:barrier}. Namely, we choose $N_0$ so large that \eqref{N0choice1} and \eqref{N0choice2} hold, and then set $T_* = 2c_0N_0^{-1}$. Moreover, the thresholds $g_1, B_1 \ge 1$ are chosen to be sufficiently large in the course of our proof, whose exact values will be summarized in Remark \ref{rmk:stokeschoice}. For the rest of this section, we assume $g \ge g_1$ and $B \ge B_1$ by default.}

We organize this section in the following way: in Section \ref{subsect:tech1}, we prove preliminary estimates of various derivatives to both $\rho$ and $u$ given bootstrap assumption \eqref{bootstrap}. In Section \ref{sec:adv} and Section \ref{sec:dtu}, we prove suitable estimates for $u\cdot\nabla u$ and $\p_t u$ respectively, which are crucial in closing the bootstrap argument. In Section \ref{subsect:closing1}, we set up a comparison scheme in between the full problem \eqref{eq:ksstokes} and the static problem \eqref{eq:ksstatstokes} so as to close the bootstrap argument.

\subsection{Preliminary Estimates}\label{subsect:tech1}
\red{For the rest of this section, we will consider a solution $(\rho,u)$ to system \eqref{eq:ksgen2} with initial data $(\rho_0, u_0)$.} We begin with several preliminary estimates for $\rho$, $u$, and some of their derivatives. Such estimates will be instrumental in later estimates of $\p_t u$ and $u \cdot \nabla u$ under the bootstrap assumption. First, we introduce a time-integrated bound for $\nabla \rho$ on any time interval of scale $T_*$.
\begin{lem}
    \label{lem:dxrho}
    For arbitrary $t_0 \in [0, \calT_0]$ and given bootstrap assumption \eqref{bootstrap}, we have
    \begin{equation}
        \label{est:dxrho}
        \int_{t_0}^{t_0 + T_*/9} \|\nabla\rho\|_{L^2}^2 dt \le 18N_0.
    \end{equation}
    \red{Here, we recall that $[0,\calT_0]$ is the bootstrap horizon.}
\end{lem}
\begin{proof}
    This lemma follows directly from applying Corollary \ref{cor:chartime} with $M = 9N_0$.
\end{proof}

We then give an estimate of $u$ that is uniform over the time interval $[0,\calT_0+T_*/9]$, which is slightly longer than the bootstrap horizon. In the mean time, we also show a bound on $\nabla u$ on an interval of scale $T_*$.
\begin{lem}
    \label{lem:supu}
    Given bootstrap assumption \eqref{bootstrap}, the following estimate holds:
    \begin{equation}
        \label{est:supu}
        \sup_{0\le t \le \calT_0+T_*/9}\|u(t)\|_{L^2}^2 \le C(\rho_m, \|u_0\|_{L^2})g^2N_0.
    \end{equation}
    Moreover, for arbitrary $t_0 \in [0, \calT_0]$, we have the following gradient bound:
    \begin{equation}
        \label{est:dxu}
        \int_{t_0}^{t_0 + T_*/9}\|\nabla u(s)\|_{L^2}^2 ds \le C(\rho_m, \|u_0\|_{L^2})g^2
    \end{equation}
\end{lem}
\begin{proof}
    Fix $t \in [0,\calT_0 + T_*/9]$. Testing \eqref{eq:ksstokesvel} by $u$ and using Cauchy-Schwarz inequality, we obtain that
    \begin{equation}\label{est:uaux1}
        \frac{1}{2}\frac{d}{dt}\|u\|_{L^2}^2 + \int_\Omega u_j\p_j u_i u_i dx + B\int_\Omega u \Delta u dx = Bg\int_\Omega \rho u_2 dx \le Bg \|\rho\|_{L^2}\|u\|_{L^2}.
    \end{equation}
    To treat the second and third term on LHS above, we use $\divv u = 0$ and $u_2|_{\p\Omega} = 0$ to deduce that
    $$
    \int_\Omega u_j\p_j u_i u_i dx = \frac{1}{2}\int_\Omega u_j\p_j |u|^2 dx = -\frac{1}{2}\int_\Omega \divv u |u|^2 dx + \int_\T u_2|u|^2\Big|_{x_2 = 0}^{x_2 = \pi} dx_1 = 0.
    $$
    Similarly,
    \begin{align*}
        \int_\Omega u\Delta u dx &= \int_\T\int_0^\pi u_j\p_i\p_i u_j dx_2 dx_1 = -\|\nabla u\|_{L^2}^2 + \int_{\T}u_j\p_2u_j\Big|_{x_2 = 0}^{x_2 = \pi} dx_1\\
        &= -\|\nabla u\|_{L^2}^2 +\int_{\T}(u_1\p_2u_1 + u_2\p_2u_2)\Big|_{x_2 = 0}^{x_2 = \pi} dx_1 = -\|\nabla u\|_{L^2}^2,
    \end{align*}
    where we used the Lions boundary condition in the final equality. Combining the two identities above with \eqref{est:uaux1}, we have
    \begin{equation} \label{est:uaux2}
        \frac{1}{2}\frac{d}{dt}\|u\|_{L^2}^2 + B\|\nabla u\|_{L^2}^2 \le Bg\|\rho\|_{L^2}\|u\|_{L^2}  
    \end{equation}
    
    On the other hand, we apply Poincar\'e inequality \eqref{est:Poincare} and apply Cauchy-Schwarz inequality to \eqref{est:uaux2}, which yields the following differential inequality:
    $$
    \frac{d}{dt}\|u\|_{L^2}^2 \le -\frac{B}{C}\|u\|_{L^2}^2 + 4CBg^2\|\rho\|_{L^2}^2.
    $$
    Further invoking Duhamel's formula, we deduce from the above estimate that:
    \begin{equation}
    \label{supu:aux1}
        \|u(t)\|_{L^2}^2 \le e^{-\frac{B}{C}t}\|u_0\|_{L^2}^2 + 4CBg^2 \int_0^t e^{-\frac{B}{C}(t-s)}\|\rho(s)\|_{L^2}^2 ds,
    \end{equation}
    for $t \in [0,\calT_0+T_*/9]$. To estimate \eqref{supu:aux1}, we first use bootstrap assumption \eqref{bootstrap} and Corollary \ref{cor:chartime} to conclude that 
    $$
    \sup_{0 \le t \le \calT_0 + T_*/9}\|\rho(t) - \rho_m\|_{L^2}^2 \le 18N_0.
    $$
    Then, we deduce from \eqref{supu:aux1} that:
    \begin{align*}
        \|u(t)\|_{L^2}^2 &\le \|u_0\|_{L^2}^2 + 4CBg^2(\sup_{0\le t \le \calT_0+T_*/9}\|\rho(t) - \rho_m\|_{L^2}^2 + 2\pi^2\rho_m^2)\int_0^t e^{-\frac{B}{C}(t-s)}ds\\
        &\le \|u_0\|_{L^2}^2 + 4CBg^2(18N_0 + 2\pi^2\rho_m^2)\cdot B^{-1}(1-e^{-\frac{Bt}{C}})\\
        &\le \|u_0\|_{L^2}^2 + C(\rho_m)g^2N_0 \le C(\rho_m, \|u_0\|_{L^2})g^2N_0
    \end{align*}
    where third inequality follows from $\rho_m^2 \le \frac{1}{2}\rho_mN_0$ due to our choice of $N_0 \ge 2\rho_m$. We also used $g, N_0 \ge 1$ in the final inequality. This concludes the proof of \eqref{est:supu}.

To show \eqref{est:dxu}, we start with \eqref{est:uaux2} to deduce that
    \begin{align*}
    \frac{1}{2}\frac{d}{dt}\|u\|_{L^2}^2 + B\|\nabla u\|_{L^2}^2 &\le C(\rho_m, \|u_0\|_{L^2})Bg(18N_0 + 2\pi^2\rho_m^2)^{1/2}gN_0^{1/2} \le C(\rho_m, \|u_0\|_{L^2})Bg^2N_0,
    \end{align*}
    where we invoked the bootstrap assumption and \eqref{est:supu} in the second inequality. Fixing $t_0 \in [0,\calT_0]$, and integrating the above from $t_0$ to $t_0 + T_*/9$, we obtain that
    \begin{align*}
        B\int_{t_0}^{t_0 + T_*/9}\|\nabla u(t)\|_{L^2}^2 dt &\le \|u_0\|_{L^2}^2 + C(\rho_m, \|u_0\|_{L^2})Bg^2N_0\cdot c_0N_0^{-1} \le C(\rho_m, \|u_0\|_{L^2})Bg^2,
    \end{align*}
    where the last inequality holds as $g, B \ge 1$. Now, the desired estimate follows from dividing on both sides of the above inequality by $B$.
\end{proof}

We conclude this subsection by an estimate on the time derivative of cell density $\rho$.
\begin{lem}
    \label{lem:dtrho}
    Assume the bootstrap assumption \eqref{bootstrap}. For any $t_0 \in [0,\calT_0]$, we have
    \begin{equation}
        \label{est:dtrho}
        \int_{t_0}^{t_0 + T_*/9} \|\p_t \rho\|_{H^{-1}_0}^2 dt \le C(\rho_m,\|u_0\|_{L^2})g^2N_0^2
    \end{equation}
\end{lem}
\begin{proof}
    Picking arbitrary $\varphi \in H^1_0$ and using \eqref{eq:ksstokesden}, we have
    \begin{align*}
        \left|\int_\Omega \p_t \rho \varphi dx \right| & \le \left|\int_\Omega u\cdot \nabla \rho \varphi dx\right| + \left|\int_\Omega \Delta \rho \varphi dx\right| + \left|\int_\Omega \varphi \divv(\rho\nabla(-\Delta_N)^{-1}(\rho - \rho_m)) dx\right| =: \sum_{i = 1}^3 I_i.
    \end{align*}
    Using incompressibility of $u$ and integrating by parts, we first see that
    $$
    I_1 = \left|\int_\Omega u\cdot \nabla \varphi (\rho-\rho_m) dx\right| \le \|\rho - \rho_m\|_{L^\infty}\|u\|_{L^2}\|\nabla\varphi\|_{L^2},
    $$
    Applying integration by parts in a similar fashion and using Sobolev embeddings as well as H\"older's inequality, we also have
    \begin{align*}
        I_2 &= \left|\int_\Omega \nabla\rho \cdot \nabla \varphi dx\right| \le \|\nabla \rho\|_{L^2}\|\nabla\varphi\|_{L^2},\\
        I_3 &= \left|\int_\Omega \rho \nabla\varphi\cdot \nabla(-\Delta_N)^{-1}(\rho - \rho_m) dx\right| \le \|\rho\|_{L^4}\|\nabla(-\Delta_N)^{-1}(\rho - \rho_m)\|_{L^4}\|\nabla\varphi\|_{L^2}\\
        &\le \|\rho\|_{L^4}\|\rho - \rho_m\|_{L^2}\|\nabla\varphi\|_{L^2} \le \|\rho\|_{L^\infty}^{1/2}\|\rho - \rho_m\|_{L^2}^{3/2}\|\nabla\varphi\|_{L^2}.
    \end{align*}
    Combining estimates of $I_i$ above and using duality, we conclude that for any $t \in [t_0, t_0 + T_*/9]$:
    \begin{align*}
    \|\p_t \rho\|_{H^{-1}_0}^2 &\lesssim \|\rho - \rho_m\|_{L^\infty}^2\|u\|_{L^2}^2 +  \|\nabla \rho\|_{L^2}^2 + \|\rho\|_{L^\infty}\|\rho - \rho_m\|_{L^2}^{3}\\
    &\le C(\rho_m,\|u_0\|_{L^2})(g^2N_0^3 + N_0^{5/2}) + \|\nabla \rho\|_{L^2}^2\\
    &\le C(\rho_m,\|u_0\|_{L^2})g^2N_0^3+ \|\nabla \rho\|_{L^2}^2,
    \end{align*}
    where we used the bootstrap assumption \eqref{bootstrap}, Proposition \ref{lem:L2toLinfty}, and Lemma \ref{lem:supu} in the second inequality. Finally, we integrate over $[t_0, t_0 + T_*/9]$ and invoke Lemma \ref{lem:dxrho} to conclude that
    $$
    \int_{t_0}^{t_0 + T_*/9}\|\p_t \rho\|_{H^{-1}_0}^2 dt \le \frac{c_0N_0^{-1}}{9}\cdot C(\rho_m,\|u_0\|_{L^2})g^2N_0^3 + 18N_0 \le C(\rho_m,\|u_0\|_{L^2})g^2N_0^2
    $$
    since $g \ge 1$ and $N_0 \ge 1$.
\end{proof}

\subsection{Control of the Advection Term}\label{sec:adv}
In this subsection, we estimate the advection term $u\cdot \nabla u$ on any time interval of length $T$ in the bootstrap horizon.

\begin{lem}
    \label{lem:advection}
    Assume the bootstrap assumption \eqref{bootstrap} holds and fix \red{arbitrary} $T \in (0,\calT_0 + T_*/9)$. There exists a constant $\calC_0 = \calC_0(\rho_m, \|u_0\|_{L^2})$ such that for $B \ge \calC_0gN_0^{1/2}$, the following estimate holds:
    \begin{equation}
        \label{est:advection}
        \int_{t_0}^{t_0 + T}\|u\cdot\nabla u (t)\|_{L^2}^2 dt \le C(\rho_m, \|u_0\|_1)g^2N_0(g^2N_0T + 1),
    \end{equation}
    for any $t_0 \in [0, \calT_0 + T_*/9 - T]$.
\end{lem}
\begin{proof}
    First using H\"older inequality and Ladyzhenskaya's inequality in 2D,  we observe that 
    \begin{align}
        \|u\cdot\nabla u\|_{L^2} &\le \|u\|_{L^4}\|\nabla u\|_{L^4} \lesssim \|u\|_{L^2}^{1/2}\|u\|_{1}\|\nabla u\|_{1}^{1/2}\notag\\
        &\le \|u\|_{L^2}^{1/2}(\|u\|_{L^2} + \|\nabla u\|_{L^2})(\|\nabla u\|_{L^2}^{1/2} + \|\nabla^2 u\|_{L^2}^{1/2})\notag\\
        &= \|u\|_{L^2}^{3/2}\|\nabla u\|_{L^2}^{1/2} + \|u\|_{L^2}^{3/2}\|\nabla^2 u\|_{L^2}^{1/2} + \|u\|_{L^2}^{1/2}\|\nabla u\|_{L^2}^{3/2} + \|u\|_{L^2}^{1/2}\|\nabla u\|_{L^2}\|\nabla^2 u\|_{L^2}^{1/2}\notag
    \end{align}
    We show how to estimate the most singular term $\|u\|_{L^2}^{1/2}\|\nabla u\|_{L^2}\|\nabla^2 u\|_{L^2}^{1/2}$, and the rest of the terms follow from a similar argument. Using the interpolation inequality $\|\nabla u\|_{L^2}^2 \le \|u\|_{L^2}\|\Delta u\|_{L^2}$ for $u \in W$, and the calculus lemma \ref{lem:ibp}, we have
    \begin{equation}
    \|u\|_{L^2}\|\nabla u\|_{L^2}^2\|\nabla^2 u\|_{L^2} \le \|u\|_{L^2}^2\|\Delta u\|_{L^2}^2. \label{est:auxadvbd}
    \end{equation}
    \red{Integrating \eqref{est:auxadvbd} in time from $t_0$ to $t_0 + T$ and using \eqref{est:supu},} we have
    \begin{align*}
    \int_{t_0}^{t_0 + T}\|u\|_{L^2}^2\|\Delta u\|_{L^2}^2 dt &\le \sup_{0 \le t \le \calT_0 + T_*/9}\|u(t)\|_{L^2}^2\int_{t_0}^{t_0 + T}\|\Delta u(t)\|_{L^2}^2 dt \le C(\rho_m,\|u_0\|_{L^2})g^2N_0\int_{t_0}^{t_0 + T}\|\Delta u(t)\|_{L^2}^2 dt,
    \end{align*}
    where we used \eqref{est:supu} in the second inequality. Proceeding similarly with other terms, we can conclude the following bound:
    \begin{equation}\label{est:adv0aux11}
    \int_{t_0}^{t_0 + T}\|u\cdot \nabla u\|_{L^2}^2 dt \le C(\rho_m,\|u_0\|_{L^2})g^2N_0\int_{t_0}^{t_0 + T}\|\Delta u\|_{L^2}^2 dt.
    \end{equation}
    Thus to close the estimate, we need to obtain a control of $\Delta u$ over an interval of length $T$. Testing \eqref{eq:ksstokesvel} by $-\Delta u$, we obtain that for any $t \in (t_0, t_0+T)$:
    \begin{align}
        \frac{1}{2}&\frac{d}{dt}\|\nabla u\|_{L^2}^2 + B\|\Delta u\|_{L^2}^2 = \int_\Omega (\Delta u)(u\cdot \nabla u) dx + Bg\int_\Omega \rho\Delta u_2 dx \notag\\
        &\le \|\Delta u\|_{L^2}\|u\cdot \nabla u\|_{L^2} + \frac{B}{4}\|\Delta u\|_{L^2}^2 + Bg^2\|\rho\|_{L^2}^2 \le C\|u\|_{L^2}\|\Delta u\|_{L^2}^2 + \frac{B}{4}\|\Delta u\|_{L^2}^2 + Bg^2\|\rho\|_{L^2}^2 \notag\\
        &\le C(\rho_m,\|u_0\|_{L^2})gN_0^{1/2}\|\Delta u\|_{L^2}^2 + \frac{B}{4}\|\Delta u\|_{L^2}^2 + Bg^2\|\rho\|_{L^2}^2, \label{est:adv0aux1}
    \end{align}
    where we used \eqref{est:auxadvbd} in the second inequality and \eqref{est:supu} in the last inequality above. Choose \begin{equation}\label{b1choice1}
        B \ge 4C(\rho_m,\|u_0\|_{L^2})gN_0^{1/2} =: \calC_0gN_0^{1/2},
    \end{equation}
    \red{where constant $C(\rho_m,\|u_0\|_{L^2})$ is the one appearing the last inequality of \eqref{est:adv0aux1}. Plugging \eqref{b1choice1} into \eqref{est:adv0aux1} and rearranging, we have}
    $$
    \frac{d}{dt}\|\nabla u\|_{L^2}^2 + B\|\Delta u\|_{L^2}^2 \le 2Bg^2\|\rho\|_{L^2}^2 \le 2Bg^2(18N_0 + 2\pi^2\rho_m^2),
    $$
    for $t \in (t_0, t_0 + T)$.
    Here, we used the bootstrap assumption \eqref{bootstrap} and Corollary \ref{cor:chartime}. 
    
    Now integrating from $t_0$ to $t_0 + T$, we obtain that
    \begin{align*}
    B\int_{t_0}^{t_0 + T}\|\Delta u(t)\|_{L^2}^2 dt &\le \|\nabla u_0\|_{L^2}^2 + C(\rho_m)Bg^2N_0T \le C(\rho_m, \|u_0\|_1)(Bg^2N_0T + 1).
    \end{align*}
    After dividing by $B$,
    $$
    \int_{t_0}^{t_0 + T}\|\Delta u(t)\|_{L^2}^2 dt \le C(\rho_m, \|u_0\|_1)(g^2N_0T + 1),
    $$
    where we used $B \ge 1$ in the inequality. Combining with \eqref{est:adv0aux11}, we obtain the desired result.
\end{proof}

\subsection{Control of \texorpdfstring{$\p_t u$}{dtu}}\label{sec:dtu}
\label{subsect:dtu}
In this section, our ultimate goal is to show \red{both a pointwise-in-time and a} time-integrated bound on the time derivative $\p_t u$. As the first step, we show an estimate of $\p_t u$ near time zero, which is presented in Lemma \ref{lem:dtu0}. \red{The fundamental reason that we prove this lemma is rather technical: we do not assume appropriate compatibility condition on initial datum $u_0$. For such general datum, one generally cannot obtain a pointwise-in-time control of $\|\p_t u(0,\cdot)\|_{L^2}$ by naively taking the limit of $t \to 0$ using the equation. (We refer to \cite{evans2022partial} for a more detailed discussion on compatible data for parabolic problems.) While a control of $\p_t u$ exactly at time zero is not available in general, we may instead control $\|\p_t u\|_{L^2}$ at some later time due to parabolic smoothing. More precisely, we have:}

\begin{lem}
    \label{lem:dtu0}
    Define a time instance $s_0 := c_0\min(\|\rho_0 - \rho_m\|_{L^2}^{-2},\rho_m^{-1})$, where $c_0$ is fixed as in Corollary \ref{cor:chartime}. \red{Then for $B \ge g\max(\calC_0N_0^{1/2},N_0)$, where $\calC_0$ is defined in Lemma \ref{lem:advection},} we have:
    \begin{equation}
    \label{est:dtu0int}
    \int_0^{s_0}\|\p_t u(t)\|_{L^2}^2 dt \le C(\rho_m, \|u_0\|_1)B^2g^2.
    \end{equation}
    In particular, there exists $\tau_0 \in (0,s_0)$ such that
    \begin{equation}
        \label{est:dtu0}
        \|\p_t u(\tau_0)\|_{L^2}^2 \le \frac{C(\rho_m, \|u_0\|_1)B^2g^2}{s_0} =: \calC_1(\rho_m, \|\rho_0 - \rho_m\|_{L^2}, \|u_0\|_1)B^2g^2.
    \end{equation}
\end{lem}

\begin{proof}
    Note that given \eqref{est:dtu0int}, \eqref{est:dtu0} directly follows from Chebyshev's inequality. Hence, we focus on showing \eqref{est:dtu0int}. Multiplying $\p_t u$ on both sides of \eqref{eq:ksstokesvel} and integrating in space, we have
    \begin{equation}
    \label{est:dtu0aux1}
        \|\p_t u\|_{L^2}^2 + B\int_\Omega \p_t u \cdot (-\Delta u) dx = Bg\int_\Omega \rho \p_t u_2 dx- \int_\Omega \p_t u (u\cdot\nabla u) dx.
    \end{equation}
    Since $u \in W$, we can further compute that
    \begin{align*}
        \frac{1}{2}\frac{d}{dt}\|\nabla u\|_{L^2}^2 &= \int_\Omega \p_t \p_iu_j \p_iu_j dx = \int_\Omega \p_t u \cdot (-\Delta u) dx + \int_{\p\Omega} \p_t u_j \p_2 u_j dS\\
        &= \int_\Omega \p_t u \cdot (-\Delta u) dx + \int_{\p\Omega} \p_t u_1 \p_2 u_1 dS + \int_{\p\Omega} \p_t u_2 \p_2 u_2 dS = \int_\Omega \p_t u \cdot (-\Delta u) dx,
    \end{align*}
    where we used the boundary conditions $u_2|_{\p\Omega} = \p_2 u_1|_{\p\Omega} = 0$ in the third inequality above. Combining with \eqref{est:dtu0aux1} and using Cauchy-Schwarz inequality, we obtain the following energy inequality:
    \begin{equation*}
        \|\p_t u\|_{L^2}^2 + \frac{B}{2}\frac{d}{dt}\|\nabla u\|_{L^2}^2 \le \frac{1}{2}\|\p_t u\|_{L^2}^2 + \frac{3}{4}\left(B^2g^2 \|\rho\|_{L^2}^2 + \|u\cdot\nabla u\|_{L^2}^2\right),\quad t \in [0,s_0].
    \end{equation*}
    Rearranging and integrating in time from $0$ to $s_0$, we have
    \begin{equation}\label{est:dtu0aux2}
    B\|\nabla u(s_0)\|_{L^2}^2 + \int_0^{s_0}\|\p_t u(t)\|_{L^2}^2 dt \le B\|\nabla u_0\|_{L^2}^2 + \frac{3}{2}\int_0^{s_0}\left(B^2g^2 \|\rho(t)\|_{L^2}^2 + \|u\cdot\nabla u(t)\|_{L^2}^2\right)dt.
    \end{equation}
    Invoking Lemma \ref{lem:advection} with $t_0 = 0$, $T = s_0$, we arrive at
    $$
    \int_{0}^{s_0}\|u\cdot\nabla u (t)\|_{L^2}^2 dt \le C(\rho_m, \|u_0\|_1)g^2N_0(g^2N_0s_0 + 1).
    $$
    Note that we can indeed invoke Lemma \ref{lem:advection} here due to $s_0 < \calT_0$ and our assumption on $B$. Combining this bound with \eqref{est:dtu0aux2}, we conclude that
    \begin{align*}
        \int_0^{s_0}\|\p_t u(t)\|_{L^2}^2 dt &\le B\|\nabla u_0\|_{L^2}^2 + C(\rho_m, \|u_0\|_1)(B^2g^2\|\rho_0 - \rho_m\|_{L^2}^2\cdot s_0 + g^4N_0^2s_0 + g^2N_0)\\
        &\le C(\rho_m, \|u_0\|_1)(B + B^2g^2 + g^4N_0^2s_0 + g^2N_0) \le C(\rho_m, \|u_0\|_1)B^2g^2,
    \end{align*}
    where the last inequality follows from choosing $B^2 \ge g^2N_0^2 \ge N_0$, as well as the fact that $s_0 
    \le \rho_m^{-1}$. Hence, we have shown \eqref{est:dtu0int}.
\end{proof}

With the $L^2$ estimate \eqref{est:dtu0} of $\p_t u$ near time zero above, we may further derive the following crucial estimates of $\p_t u$ up to $\calT_0$ thanks to the bootstrap assumption \eqref{bootstrap}. More precisely:
\begin{lem}
\label{lem:dtu}
    Assume the bootstrap assumption \eqref{bootstrap} holds. There exists $\tau_0 < s_0$ and \red{$\calB(\rho_m, \|u_0\|_1)$ large so that for any $B \ge B_0g^2$,}
    \begin{equation}
        \label{est:supdtu}
        \sup_{\tau_0 \le t \le \calT_0}\|\p_t u(t)\|_{L^2}^2 \le 5\calC_1(\rho_m, \|\rho_0 - \rho_m\|_{L^2}, \|u_0\|_1)B^2g^2,
    \end{equation}
    \red{where $\calC_1$ is defined in \eqref{est:dtu0}.} Moreover, for any $t_0 \in [\tau_0, \calT_0]$, we have
    \begin{equation}
        \label{est:l2dtu}
        \int_{t_0}^{t_0 + T_*/9}\|\p_t u(t)\|_{L^2}^2 dt \le C(\rho_m, \|u_0\|_1)Bg^2.
    \end{equation}
    \red{Additionally, if $t_0 \in [\tau_0, \calT_0 - \frac{8}{9}T_*]$, then
    \begin{equation}\label{est:l2dtu1}
        \int_{t_0}^{t_0 + T_*}\|\p_t u(t)\|_{L^2}^2 dt \le C(\rho_m, \|u_0\|_1)Bg^2.
    \end{equation}
    }
\end{lem}
\begin{rmk}
    We remark that \eqref{est:l2dtu} does not follow directly from integrating \eqref{est:supdtu} in time. In fact, doing so will only lead to an insufficient upper bound of size $B^2g^2N_0^{-1}$.
\end{rmk}
\begin{proof}
    The proof for this lemma is rather lengthy, so we split the proof in several steps. \\
    \textbf{Step 1: Energy Estimate for $\p_t u$.} We start with the evolution of $\p_t u$. By differentiating \eqref{eq:ksstokesvel} in time, it is straightforward to check that $\p_t u$ satisfies the following equation:
    $$
    \p_t^2 u + \p_t u \cdot \nabla u + u\cdot \nabla \p_t u - B\Delta\p_t u + \nabla \p_t p = Bg(\p_t\rho)(0,1)^T.
    $$
Testing by $\p_t u$, using $\divv(\p_t u) = 0$ and the fact that $\p_t u$ still satisfies the Lions boundary condition, we have for arbitrary $\epsilon > 0$ that
\begin{equation}\label{dtuaux1}
\begin{split}
        \frac{1}{2}\frac{d}{dt}\|\p_t u\|_{L^2}^2 &+ B\|\nabla \p_t u\|_{L^2}^2 = -\int_\Omega (\p_t u_j) (\p_j u_i) (\p_t u_i) dx + Bg\int_\Omega \p_t \rho \p_t u_2 dx \\
        &\le -\int_\Omega (\p_t u_j) (\p_j u_i) (\p_t u_i) dx + \epsilon B\|\p_t u_2\|_{H^1_0}^2 + C(\epsilon)Bg^2 \|\p_t \rho\|_{H^{-1}_0}^2 \\
        &\le -\int_\Omega (\p_t u_j) (\p_j u_i) (\p_t u_i) dx + \epsilon CB\|\nabla \p_t u\|_{L^2}^2 + C(\epsilon)Bg^2 \|\p_t \rho\|_{H^{-1}_0}^2,
    \end{split}
\end{equation}
where we used the $H^1_0$--$H^{-1}_0$ duality in the first inequality, and Poincar\'e inequality in the final inequality. The use of duality is indeed valid since the Lions boundary condition implies that $u_2$, and thus $\p_t u_2$, satisfies zero Dirichlet boundary condition. 

To estimate the first term on the RHS of the estimate above, an application of Ladyzhenskaya's inequality in 2D and Cauchy-Schwarz inequality gives:
\begin{equation}\label{dtuaux2}
\begin{split}
\int_\Omega (\p_t u_j) (\p_j u_i) (\p_t u_i) dx &\le \|\p_t u\|_{L^4}^2\|\nabla u\|_{L^2} \le C\|\p_t u\|_{L^2}(\|\p_t u\|_{L^2} + \|\nabla \p_t u\|_{L^2})\|\nabla u\|_{L^2}\\
&\le \frac{B}{4}\|\nabla \p_t u\|_{L^2}^2 + \frac{C}{B}\|\p_t u\|_{L^2}^2\|\nabla u\|_{L^2}^2 + C\|\p_t u\|_{L^2}^2\|\nabla u\|_{L^2}.
\end{split}
\end{equation}

Now choosing $\epsilon > 0$ sufficiently small in \eqref{dtuaux1}, combining with \eqref{dtuaux2}, and rearranging, we have the following differential inequality:
\begin{equation}
    \label{est:stokesptu}
    \begin{split}
    \frac{d}{dt}\|\p_t u\|_{L^2}^2 &\le -B\|\nabla \p_t u\|_{L^2}^2 + \frac{C}{B}\|\p_t u\|_{L^2}^2\|\nabla u\|_{L^2}^2 +C\|\p_t u\|_{L^2}^2\|\nabla u\|_{L^2} + CBg^2 \|\p_t\rho\|_{H^{-1}_0}^2\\
    &\le \left(\frac{C}{B}\|\nabla u\|_{L^2}^2 + C\|\nabla u\|_{L^2} -\frac{B}{C_P}\right)\|\p_tu\|_{L^2}^2 + CBg^2\|\p_t\rho\|_{H^{-1}_0}^2,
    \end{split}
\end{equation}
where we used the Poincar\'e inequality \eqref{est:Poincare} in the second inequality.

Choose $\tau_0$ as in Lemma \ref{lem:dtu0}. We write \eqref{est:stokesptu} in Duhamel form and use \eqref{est:dtu0} to obtain that, for any $t \in [\tau_0,\tau_0 + T_*/9]$,
    \begin{equation}
    \label{est:dtuduhamel1}
    \begin{split}
    \|\p_t u (t)\|_{L^2}^2 &\le \calC_1B^2g^2\exp\left(-\frac{B(t - \tau_0)}{C_P} + C\int_{\tau_0}^t \left(\frac{1}{B}\|\nabla u(s)\|_{L^2}^2 + \|\nabla u(s)\|_{L^2} \right)ds\right)\\
    &\quad+ CBg^2\int_{\tau_0}^t \exp\left(-\frac{B(t-s)}{C_P} + C\int_s^t \left(\frac{1}{B}\|\nabla u(\tau)\|_{L^2}^2 + \|\nabla u(\tau)\|_{L^2}\right) d\tau\right)\|\p_t \rho (s)\|_{H^{-1}_0}^2 ds\\
    &\le \calC_1B^2g^2\exp\left(-\frac{B(t - \tau_0)}{C_P} + C(\rho_m, \|u_0\|_1)\left(\frac{g^2}{B} + g(t - \tau_0)^{1/2}\right)\right)\\
    &\quad+ CBg^2\int_{\tau_0}^t \exp\left(-\frac{B(t-s)}{C_P} + C(\rho_m, \|u_0\|_1)\left(\frac{g^2}{B} + g(t - s)^{1/2}\right)\right) \|\p_t \rho (s)\|_{H^{-1}_0}^2 ds.
    \end{split}
    \end{equation}
    Here, we used \eqref{est:dxu} in the second inequality above. We also utilized the following consequence of \eqref{est:dxu}: for any $\tau_0 \le a \le b \le a + T_*/9$,
    $$
    \int_a^b \|\nabla u(s)\|_{L^2} ds \le (b-a)^{1/2}\left(\int_{a}^{a + T_*/9} \|\nabla u(s)\|_{L^2}^2 ds\right)^{1/2} \le C(\rho_m, \|u_0\|_1)g(b-a)^{1/2}.
    $$
    Then, we choose $B$ sufficiently large that
    \begin{equation}\label{bcond2}
    B \ge \frac{\calC_2}{\log 2}g^2,
    \end{equation}
    \red{where $\calC_2 := C(\rho_m, \|u_0\|_1)$ is the constant appeared in \eqref{est:dtuduhamel1}.} This reduces \eqref{est:dtuduhamel1} to
    \begin{equation}
    \label{est:dtuduhamel2}
    \begin{split}
    \|\p_t u (t)\|_{L^2}^2 &\le 2\calC_1B^2g^2 \exp\left({-\frac{B(t-\tau_0)}{C_P}} + \calC_2 g(t-\tau_0)^{1/2}\right)\\
    &\quad+ CBg^2\int_{\tau_0}^t \exp\left(-\frac{B(t-s)}{C_P} +  \calC_2 g(t-s)^{1/2}\right)\|\p_t \rho (s)\|_{H^{-1}_0}^2 ds =: I_1 + I_2.
    \end{split}
    \end{equation}
    This concludes the proof of energy estimates that we will use in later steps.

\noindent\textbf{Step 2: Proof of \eqref{est:supdtu}.} {The plan of proving \eqref{est:supdtu} is as follows: we first show a uniform bound similar to the desired control \eqref{est:supdtu}, yet on a shorter time interval, together with an improved bound at the endpoint of this interval. Namely, we will show that
\begin{equation}
    \label{est:supdtushort1}
    \sup_{\tau_0 \le t \le \tau_0 + T_*/9}\|\p_t u(t)\|_{L^2}^2 \le 5\calC_1B^2g^2,
\end{equation}
and additionally
\begin{equation}
    \label{est:supdtushort2}
    \|\p_t u(\tau_0 + T_*/9)\|_{L^2}^2 \le \calC_1B^2g^2 = \|\p_t u(\tau_0)\|_{L^2}^2.
\end{equation}
With \eqref{est:supdtushort1} and \eqref{est:supdtushort2}, we can prove \eqref{est:supdtu} using an iteration argument. 
}

\noindent\textit{\textbf{Step 2.1: Proof of \eqref{est:supdtushort1} and \eqref{est:supdtushort2}}.} To prove the desired estimates \eqref{est:supdtushort1} and \eqref{est:supdtushort2}, we study the inequality \eqref{est:dtuduhamel2}. The key point is that we have to estimate \eqref{est:dtuduhamel2} in two different time scales. \red{Roughly speaking, we will choose $B^{-1} \ll T_*$ i.e. the dissipation time for $u$ being much smaller than that for $\rho$. In this scenario, $\|\p_t u\|_{L^2}$ may grow on time interval $[\tau_0, \tau_0 + B^{-1}]$, but it will not grow much due to smallness of $B^{-1}$. Yet on $[\tau_0 + B^{-1}, \tau_0 + T_*/9]$, the dissipation effect becomes significant and is sufficient to damp growth generated by nonlinear terms.}
    
    More precisely, let us choose $B$ such that 
    \begin{equation}\label{bcond3}
    \begin{split}
    \frac{100C_P}{B} <  \frac{c_0}{9N_0}=\frac{T_*}{9},&\quad B \ge \frac{100C_P\calC_2^2}{(\log2)^2}g^2,
    \end{split}
    \end{equation}
    \red{where $C_P$ is the Poincar\'e constant appeared in \eqref{est:dtuduhamel2}.} From \eqref{bcond3}, we obtain
    \begin{equation}\label{dtuaux11}
    \exp\left({-\frac{B(t-a)}{C_P}} + \calC_2 g(t-a)^{1/2}\right) \le \exp\left({-\frac{B(t-a)}{C_P}}\right)\exp\left(\frac{10\sqrt{C_P}\calC_2 g}{\sqrt{B}}\right) \le 2,
    \end{equation}
    for any time $a$ so that $0 \le t-a \le \frac{100C_P}{B}$.\\
    
    \noindent\textbf{Estimate on $[\tau_0, \tau_0 + \frac{100C_P}{B}]$.} If $t \in [\tau_0, \tau_0 + \frac{100C_P}{B}]$, we may invoke \eqref{dtuaux11} to estimate $I_1$ and $I_2$ by:
    $$
    I_1 \le 4\calC_1B^2g^2,\; I_2 \le CBg^2\int_{\tau_0}^{\tau_0 + T_*/9}\|\p_t \rho(s)\|_{H^{-1}_0}^2 ds \le C(\rho_m, \|u_0\|_{L^2})Bg^4N_0^2,
    $$
    where we used Lemma \ref{lem:dtrho} in the last inequality. Choosing $B$ sufficiently large that 
    \begin{equation}\label{bcond4}
    C(\rho_m, \|u_0\|_{L^2})g^2N_0^2 < \frac{\calC_1B}{2},
    \end{equation}
    where $C(\rho_m, \|u_0\|_{L^2})$ is the constant appearing in the previous estimate. We then conclude that
  \begin{equation}\label{comeback0}
  \sup_{\tau_0 \le t \le \tau_0 + \frac{100C_P}{B}}\| \p_t u(t)\|_{L^2}^2 < 5\calC_1B^2g^2.
  \end{equation}
    
    \noindent\textbf{Estimate on $(\tau_0 + \frac{100C_P}{B}, \tau_0 + T_*/9]$.} Suppose that $t \in (\tau_0 + \frac{100C_P}{B}, \tau_0 + T_*/9]$. In this case, we further choose $B$ large so that
    \begin{equation}\label{bcond5}
    B \ge \frac{(2C_P\calC_2)^2}{100C_P}g^2.
    \end{equation}
    Such choice guarantees that for any time $a \in [\tau_0, t - \frac{100C_P}{B}]$, we have
    \begin{equation}\label{dtuaux22}
    -\frac{B(t-a)}{C_P} + \calC_2 g(t-a)^{1/2} \le -\frac{B(t-a)}{2C_P} \le -50.
    \end{equation}
    Therefore using \eqref{dtuaux22}, we can bound $I_1$ and $I_2$ in \eqref{est:dtuduhamel2} by
    \begin{equation}\label{dtuI1aux1}
    I_1 \le 2\calC_1B^2g^2\exp\left(-\frac{B(t-\tau_0)}{2C_P}\right) \le e^{-50}\calC_1B^2g^2 < \frac{1}{4}\calC_1B^2g^2.
    \end{equation}
    \begin{equation}\label{dtuI2aux1}
    \begin{split}
        I_2 &= CBg^2\left(\int_{\tau_0}^{t-\frac{100C_P}{B}} + \int_{t-\frac{100C_P}{B}}^{t}\right) \exp\left(-\frac{B(t-s)}{C_P} +  \calC_2 g(t-s)^{1/2}\right)\|\p_t \rho (s)\|_{H^{-1}_0}^2 ds\\
        &\le CBg^2(e^{-50} + 2)\int_{\tau_0}^{\tau_0 + T_*/9}\|\p_t \rho(s)\|_{H^{-1}_0}^2 ds \le C(\rho_m, \|u_0\|_{L^2})Bg^4N_0^2,
    \end{split}
    \end{equation}
    where we used \eqref{dtuaux22} to estimate the first integral and \eqref{dtuaux11} to estimate the second integral in the first inequality above. \red{Then we choose $B$ even larger that
    \begin{equation}
        \label{bcond6}
        C(\rho_m, \|u_0\|_{L^2})g^2N_0^2 \le \frac{\calC_1}{2}B,
    \end{equation}
    where $C(\rho_m, \|u_0\|_{L^2})$ is the constant appearing in the last line of \eqref{dtuI2aux1}. Combining \eqref{bcond6} with \eqref{dtuI2aux1}, we conclude that $I_2$ in \eqref{est:dtuduhamel2} can be bounded as follows:
    \begin{equation}\label{dtuI2aux2}
    I_2 \le \frac{\calC_1}{2}B^2g^2.
    \end{equation}
    }
    Finally, we combine \eqref{dtuI1aux1} and \eqref{dtuI2aux2} to deduce that
    \begin{equation}\label{comeback}
        \sup_{\tau_0 + \frac{100C_P}{B} \le t \le \tau_0 + T_*/9}\|\p_t u(t)\|_{L^2}^2 < \calC_1B^2g^2.
    \end{equation}
    \red{Collecting both \eqref{comeback0} and \eqref{comeback} directly yields \eqref{est:supdtushort1} and \eqref{est:supdtushort2}.}
    
    \noindent\textit{\textbf{Step 2.2: Iterative Procedure.}} \red{With \eqref{est:supdtushort1} and \eqref{est:supdtushort2}, we may start to iterate Step 2.1 and 2.2 for subsequent time intervals $[\tau_0 + iT_*/9, \tau_0 + (i+1)T_*/9]$, $i \ge 1$. Notice that all estimates involved in Step 2.1 and 2.2 only used that $[\tau_0, \tau_0 + T_*/9] \subset [0, \calT_0 + T_*/9]$ and $\|\p_t u(\tau_0)\|_{L^2}^2 \le \calC_1 B^2g^2$. Hence, we may define $i_*$ to be the largest natural number so that $[\tau_0 + i_*T_*/9, \tau_0 + (i_* + 1)T_*/9] \subset [0,\calT_0 + T_*/9]$. The above arguments will conclude that
    $$
    \sup_{\tau_0 \le t \le \tau_0 + (i_* + 1)T_*/9} \|\p_t u (t)\|_{L^2}^2 \le 5\calC_1 B^2g^2.
    $$
    We may without loss of generality consider $\calT_0 < \infty$. In this case, by definition of $i_*$, we must have $\tau_0 + (i_* + 1)T_*/9 > \calT_0$, since otherwise $[\tau_0 + (i_* + 1)T_*/9, \tau_0 + (i_* + 2)T_*/9] \subset [0,\calT_0 + T_*/9]$, contradicting the maximality of $i_*$. We may then conclude that
    $$
    \sup_{\tau_0 \le t \le \calT_0} \|\p_t u (t)\|_{L^2}^2 \le \sup_{\tau_0 \le t \le \tau_0 + (i_* + 1)T_*/9} \|\p_t u (t)\|_{L^2}^2 \le 5\calC_1 B^2g^2.
    $$
    This concludes the proof of \eqref{est:supdtu}, after we choose threshold $\calB$ so large that \eqref{bcond2}, \eqref{bcond3}, \eqref{bcond4}, \eqref{bcond5}, and \eqref{bcond6} hold for all $B \ge \calB g^2$. Note that we can indeed choose such $\calB$, since all conditions of choosing $B$ assume the form $B \ge Cg^2$ with some $C$.
    }
    
    \noindent\textbf{Step 3: Proof of \eqref{est:l2dtu}.} We write \eqref{est:dtuduhamel2} in Duhamel form starting at time $t_0 \in [\tau_0,\calT_0]$:
    \begin{equation}
    \label{dtuaux44}
    \begin{split}
        \int_{t_0}^{t_0 + T_*/9}\|\p_t u(t)\|_{L^2}^2 dt &\le 2\|\p_t u(t_0)\|_{L^2}^2\int_{t_0}^{t_0 + T_*/9} \exp{\left(-\frac{B}{C_P}(t-t_0) + \calC_2 g(t-t_0)^{1/2}\right)} dt\\
        &\quad+ CBg^2\int_{t_0}^{t_0 + T_*/9}\int_{t_0}^t \exp{\left(-\frac{B}{C_P}(t-s) + \calC_2 g(t-s)^{1/2}\right)}\|\p_t\rho(s)\|_{H^{-1}_0}^2ds dt\\
        &\le 10\calC_1B^2g^2\int_{t_0}^{t_0 + T_*/9} \exp{\left(-\frac{B}{C_P}(t-t_0) + \calC_2 g(t-t_0)^{1/2}\right)} dt\\
        &\quad+ CBg^2\int_{t_0}^{t_0 + T_*/9}\int_{t_0}^t \exp{\left(-\frac{B}{C_P}(t-s) + \calC_2 g(t-s)^{1/2}\right)}\|\p_t\rho(s)\|_{H^{-1}_0}^2ds dt.
        \end{split}
    \end{equation}
    Here, we used \eqref{est:supdtu} in the second inequality. Since \eqref{bcond3} and \eqref{bcond5} hold, we may run a similar argument in deducing \eqref{dtuaux11} and \eqref{dtuaux22} to conclude that
    \begin{equation}
        \label{dtuaux33}
        \begin{split}
        \exp{\left(-\frac{B}{C_P}(t-s) + \calC_2 g(t-s)^{1/2}\right)} &\le \begin{cases}
            2\exp{\left(-\frac{B(t-s)}{C_P}\right)},& 0 \le t-s \le \frac{100C_P}{B},\\
            \exp{\left(-\frac{B(t-s)}{2C_P}\right)},& \frac{100C_P}{B} < t-s \le \frac{T_*}{9},
        \end{cases}\\
        &\le 2\exp{\left(-\frac{B(t-s)}{2C_P}\right)},
        \end{split}
    \end{equation}
    for $0 \le t-s \le T_*/9$. Applying \eqref{dtuaux33} to \eqref{dtuaux44}, we have
    \begin{align*}
        \int_{t_0}^{t_0 + T_*/9}\|\p_t u(t)\|_{L^2}^2 dt &\le  C(\rho_m, \|u_0\|_1)B^2g^2\frac{C_P}{B}\left(1- e^{\frac{-BT_*}{18C_P}}\right)\\
        &\quad+ CBg^2 \int_{t_0}^{t_0 + T_*/9}\left(\int_{t_0}^se^{-\frac{B(t-s)}{2C_P}}dt\right)\|\p_t\rho(s)\|_{H^{-1}_0}^2dt\\
        &\le  C(\rho_m, \|u_0\|_1)Bg^2 + Cg^2\int_{t_0}^{t_0 + T_*/9}\|\p_t\rho(s)\|_{H^{-1}_0}^2dt\\
        &\le C(\rho_m, \|u_0\|_1)(Bg^2 + g^4N_0^2) \le C(\rho_m, \|u_0\|_1)Bg^2,
    \end{align*}
    if we choose $B$ sufficiently large that
    \begin{equation}\label{bcond7}
        B \ge g^2N_0^2.
    \end{equation}
    Note that we used \eqref{est:supdtu} and Fubini theorem in the first inequality, and Lemma \ref{lem:dtrho} in the third inequality. This concludes the proof of \eqref{est:l2dtu} after we choose threshold $\calB$ so that \eqref{bcond2}, \eqref{bcond3}, \eqref{bcond4}, \eqref{bcond5}, \eqref{bcond6}, and \eqref{bcond7} hold \red{for all $B \ge B_0g^2$}.   

    \noindent\textbf{Step 4: Proof of \eqref{est:l2dtu1}.} \red{Given $t_0 \in [\tau_0, \calT_0-\frac{8}{9}T_*]$, we note that \eqref{est:l2dtu1} follows directly from \eqref{est:l2dtu}:
    \begin{align*}
        \int_{t_0}^{t_0 + T_*}\|\p_t u(t)\|_{L^2}^2 dt &= \sum_{i = 0}^{8}\int_{t_0 + iT_*/9}^{t_0 + (i+1)T_*/9}\|\p_t u(t)\|_{L^2}^2 dt \le C(\rho_m, \|u_0\|_1)Bg^2,
    \end{align*}
    where we applied \eqref{est:l2dtu} for each term in the inequality above. }
\end{proof}

\subsection{Closing the Bootstrap}
\label{subsect:closing1}
In this section, we show the improved estimate as stated in Proposition \ref{prop:bootstrap}. Consider a time instance $\calT_1$ to be the first time such that $\|\rho - \rho_m\|_{L^2}^2$ achieves $N_0$, \red{where we recall that $N_0$ was chosen to satisfy \eqref{N0choice1} and \eqref{N0choice2}.} Note that $\calT_1$ verifies the following two properties:
\begin{enumerate}
    \item $\calT_1 > s_0$, where $s_0$ is chosen as in Lemma \ref{lem:dtu0};
    \item $\calT_1 + T_* < \calT_0$.
\end{enumerate}
The first property is due to $N_0 > 2\|\rho_0 - \rho_m\|_{L^2}^2$ and Corollary \ref{cor:chartime}; the second property is due to \red{the choice that $T_* = 2c_0N_0^{-1}$.}

Now, we set up the following comparison scheme: let $\rho_s, u_s$ be the (global) regular solution to the following static problem
 \begin{equation}
     \label{eq:ksstatstokes2}
     \begin{cases}
     \p_t \rho_s + u_s\cdot \nabla \rho_s -\Delta \rho_s + \divv(\rho_s\nabla(-\Delta_N)^{-1}(\rho_s - (\rho_s)_m)) = 0,\\
     -\Delta u_s + \nabla p_s = g\rho_s(0,1)^T,\quad \divv u_s = 0,\\
    \p_2 \rho_s\big|_{\p \Omega} = 0,\; (u_s)_2\big|_{\p \Omega} = \p_2(u_s)_1\big|_{\p \Omega} = 0,\; \rho_s(0,x) = \rho(\calT_1,x) \ge 0. 
    \end{cases}
 \end{equation}
 Define $r(t,x) = \rho(t,x) - \rho_s(t-\calT_1, x)$, $v(t,x) = u(t,x) - u_s(t-\calT_1,x)$ for $t \in [\calT_1, \calT_1 + T_*]$. A straightforward computation yields the following equations for $r$ and $v$:
\begin{subequations}
        \label{eq:diffeq}
        \begin{equation}
        \label{eq:ksstokesreq}
            \p_t r - \Delta r + v\cdot \nabla \rho + u_s\cdot\nabla r + \divv(r\nabla(-\Delta_N)^{-1}(\rho - \rho_m)) + \divv(\rho_s\nabla(-\Delta_N)^{-1}r) = 0,
        \end{equation}
        \begin{equation}
            \label{eq:ksstokesveq}
            -\Delta v + \nabla(B^{-1}p - p_s) = -\frac{1}{B}\p_t u - \frac{1}{B}(u\cdot \nabla u) + gr(0,1)^T,\; \divv v = 0,
        \end{equation}
        \begin{equation}
            r(\calT_1,x) = 0,\; v(\calT_1,x) = u(\calT_1,x)-u_s(\calT_1,x),
        \end{equation}
        \begin{equation}
        \label{eq:bdry1}
            \p_2 r|_{\p \Omega} = 0,\; v_2|_{\p \Omega} = \p_2v_1|_{\p\Omega} = 0.
        \end{equation}
    \end{subequations}
    Notice that in the derivation of \eqref{eq:ksstokesreq}, we used the fact that
    $$
    \rho(t,x) - \rho_m - (\rho_s(t-\calT_1, x) - (\rho_s)_m) = \rho(t,x) - \rho_s(t-\calT_1, x) = r(t,x),
    $$
    thanks to the observation that
    $
    (\rho_s)_m = (\rho(\calT_1,\cdot))_m = \rho_m.
    $
    \red{Our plan is to establish appropriate controls of $\|r(t)\|_{L^2}$ and $\|\rho_s(t - \calT_1) - \rho_m\|_{L^2}$ on time interval $[\calT_1, \calT_1 + T_*]$.}
    
    \noindent\textbf{Step 1: Control of $\|r(t)\|_{L^2}$.} To control the remainder $r$, we study \eqref{eq:ksstokesreq} by testing this equation by $r$. By incompressibility of $u_s$, we deduce that for any $t \in [\calT_1,\calT_1 + T_*]$:
    \begin{align*}
        \frac{1}{2}\frac{d}{dt}\|r\|_{L^2}^2 + \|\nabla r\|_{L^2}^2 &= -\int_\Omega rv\cdot \nabla \rho - \int_\Omega r \divv(r\nabla(-\Delta_N)^{-1}(\rho - \rho_m)) - \int_\Omega r \divv(\rho_s\nabla(-\Delta_N)^{-1}r)\\
        &= K_1 + K_2 + K_3.
    \end{align*}
    We first treat $K_2, K_3$. Integrating by parts, we have
    \begin{align}\label{k2}
        K_2 &= \frac{1}{2}\int_\Omega \nabla (r^2) \cdot \nabla(-\Delta_N)^{-1}(\rho - \rho_m) = \frac{1}{2}\int_\Omega r^2(\rho - \rho_m)dx \le \frac{1}{2}\|\rho - \rho_m\|_{L^\infty}\|r\|_{L^2}^2.
    \end{align}
    To estimate $K_3$, we integrate by parts and obtain
    \begin{align}\label{k3}
    K_3 &= \int_\Omega \rho_s \nabla r \cdot \nabla(-\Delta_N)^{-1}r dx \le \|\nabla r\|_{L^2}\|\rho_s\|_{L^4}\|\nabla(-\Delta_N)^{-1}r\|_{L^4}\notag\\
    &\le \frac14 \|\nabla r\|_{L^2}^2 + \frac34\|\rho_s\|_{L^4}^2\|\nabla(-\Delta_N)^{-1}r\|_{L^4}^2\notag\\
    &\le \frac14 \|\nabla r\|_{L^2}^2 + \frac34\|\rho_s\|_{L^4}^2\|r\|_{L^2}^2,
    \end{align}
    where we used the Sobolev embedding $H^1 \subset L^4$ and elliptic estimates. Finally, we treat $K_1$, which potentially contributes the growth of $\|r\|_{L^2}$ due to the appearance of $v$ in the term. In fact, we have
    \begin{align*}
        K_1 &= -\int_\Omega rv\cdot \nabla \rho dx = \int_\Omega \divv(rv)\rho dx - \int_{\p\Omega} r v_2\rho d\sigma\\
        &= \int_\Omega \nabla r \cdot v \rho dx \le \|\nabla r\|_{L^2}\|v\|_{L^\infty} \|\rho\|_{L^2} \le \frac14\|\nabla r\|_{L^2}^2 + \frac38\|v\|_2^2\|\rho\|_{L^2}^2,
    \end{align*}
    where we used $v_2|_{\p\Omega} = 0$ in the third equality, and Sobolev embedding $H^2 \subset L^\infty$ in the final inequality. 
    
    Now, we claim that $-B^{-1}\p_t u(t,\cdot) - B^{-1}u\cdot\nabla u + gr(t,\cdot)(0,1)^T \in H$. Indeed, as $u(t,\cdot) \in W$ by Theorem \ref{thm:lwpstokes}, we have $\int_\Omega u_1(x) dx = 0$. This leads to
    $$
    \int_\Omega \left(-B^{-1}\p_t u + gr(0,1)^T\right)_1 dx = \int_\Omega -B^{-1}\p_t u_1 dx = -B^{-1}\p_t\int_\Omega u_1(x)dx = 0.
    $$
    Moreover, we observe that using Lions boundary condition satisfied by $u$,
    \begin{align*}
        \int_\Omega (u\cdot\nabla u)_1 dx = \int_\Omega u_j\p_j u_1 dx = -\int_\Omega \divv uu_1 dx + \int_{\p\Omega} u_2 u_1 dS(x) =0.
    \end{align*}
    The two computations above readily verify the claim. Hence, we are eligible to apply the elliptic estimate \eqref{est:elliptic1} and obtain
    $$
    \|v\|_{2}^2 \lesssim \frac{1}{B^2}(\|\p_t u\|_{L^2}^2 + \|u\cdot\nabla u\|_{L^2}^2) + g^2\|r\|_{L^2}^2.
    $$
    We therefore estimate $K_1$ by
    \begin{equation}
        \label{k1}
        K_1 \le \frac14\|\nabla r\|_{L^2}^2 + C\left(\frac{1}{B^2}(\|\p_t u\|_{L^2}^2 + \|u\cdot\nabla u\|_{L^2}^2) + g^2\|r\|_{L^2}^2\right)\|\rho\|_{L^2}^2
    \end{equation}
    Combining \eqref{k1}, \eqref{k2}, and \eqref{k3}, we have the following differential inequality: for any $t \in [\calT_1, \calT_1 + T_*]$,
    \begin{align*}
        \frac{d}{dt}\|r\|_{L^2}^2 + \|\nabla r\|_{L^2}^2 &\le C\left(g^2\|\rho\|_{L^2}^2 + \|\rho - \rho_m\|_{L^\infty} + \|\rho_s\|_{L^4}^2\right)\|r\|_{L^2}^2 + \frac{C}{B^2}\|\rho\|_{L^2}^2\left(\|\p_t u\|_{L^2}^2 + \|u\cdot\nabla u\|_{L^2}^2\right)\\
        &\le F(t)\|r\|_{L^2}^2 + \frac{C}{B^2}(\|\rho-\rho_m\|_{L^2}^2 + 2\pi^2\rho_m^2)\left(\|\p_t u\|_{L^2}^2 + \|u\cdot\nabla u\|_{L^2}^2\right),
    \end{align*}
    where we set $F(t) = C( g^2\|\rho(t)\|_{L^2}^2 + \|\rho(t) - \rho_m\|_{L^\infty} + \|\rho_s(t - \calT_1)\|_{L^4}^2)$. Note that by the bootstrap assumption, Proposition \ref{lem:L2toLinfty}, Proposition \ref{prop:barrier}, and $g \ge 1$, we have
    $$
    \int_{\calT_1}^{\calT_1 + T_*} F(s) ds \le C(\rho_m, \|\rho_0\|_{L^2}) g^2.
    $$

    Now we would like to estimate $\|r(t)\|_{L^2}^2$ on $[\calT_1, \calT_1 + T_*]$. Since $r(\calT_1,x) = 0$, an application of Gr\"onwall inequality and bootstrap assumption \eqref{bootstrap} yields that for any $t \in [\calT_1, \calT_1 + T_*]$:
    \begin{align}
    \label{est:renergy1}
    \|r(t)\|_{L^2}^2 &\le \frac{C}{B^2}\int_{\calT_1}^{t} (\|\rho(s)-\rho_m\|_{L^2}^2 + 2\pi^2\rho_m^2)\left(\|\p_t u(s)\|_{L^2}^2 + \|u\cdot\nabla u(s)\|_{L^2}^2\right) ds \left[ \exp\left(\int_{\calT_1}^{t} F(s)ds\right)\right] \notag\\
    &\le C(\rho_m, \|\rho_0 - \rho_m\|_{L^2})e^{g^2}\cdot \frac{1}{B^2}\int_{\calT_1}^{\calT_1 + T_*}\left(\|\p_t u(s)\|_{L^2}^2 + \|u\cdot\nabla u(s)\|_{L^2}^2\right) ds
    \end{align}
    Choose
    \begin{equation}
        \label{b1}
        B \ge \max(\calB g^2, \calC_0gN_0^{1/2}),
    \end{equation}
    where $\calB$ and $\calC_0$ are chosen as in Lemma \ref{lem:dtu} and Lemma \ref{lem:advection} respectively. Then we may invoke \eqref{est:l2dtu1} (since $\calT_1 + T_* < \calT_0$ and $\calT_1 > s_0 \ge \tau_0$) and Lemma \ref{lem:advection} to estimate that
    \begin{align*}
    \int_{\calT_1}^{\calT_1 + T_*}\|\p_t u(s)\|_{L^2}^2 + \|u\cdot\nabla u(s)\|_{L^2}^2 ds &\le C(\rho_m, \|u_0\|_{1})(g^2B + g^4N_0^2T_* + g^2N_0)\\
    &= C(\rho_m, \|u_0\|_{1})(g^2B + g^4N_0 + g^2N_0),
    \end{align*}
    where we used the fact that $T_* = 2c_0N_0^{-1}$ in the last equality. Combining this estimate with \eqref{est:renergy1}, \red{we obtain
    \begin{equation}
        \label{est:rest1}
        \begin{split}
        \sup_{\calT_1 \le t \le \calT_1 + T_*}\|r(t)\|_{L^2}^2 &\le C(\rho_m, \|\rho_0-\rho_m\|_{L^2},\|u_0\|_1)\frac{(g^2B + g^4N_0 + g^2N_0)e^{g^2}}{B^2}\\
        &\le C(\rho_m, \|\rho_0-\rho_m\|_{L^2},\|u_0\|_1)\frac{g^2e^{g^2}}{B},
        \end{split}
    \end{equation}
    where we used $g \ge 1$ and \eqref{b1} in the last inequality above. Now choosing
    \begin{equation}
        \label{b2}
        B \ge \frac{4C(\rho_m, \|\rho_0-\rho_m\|_{L^2},\|u_0\|_1)}{N_0}g^2e^{g^2},
    \end{equation}
    where $C(\rho_m, \|\rho_0-\rho_m\|_{L^2},\|u_0\|_1)$ is the constant appearing in the last line of \eqref{est:rest1}, we deduce that
    \begin{equation}
        \label{est:rest}
        \sup_{\calT_1 \le t \le \calT_1 + T_*}\|r(t)\|_{L^2}^2 \le \frac{N_0}{4}.
    \end{equation}
    }
    
    \noindent\textbf{Step 2: Control of $\|\rho_s(t - \calT_1) - \rho_m\|_{L^2}$.} {Since $\rho_s$ satisfies the static problem \eqref{eq:ksstatstokes2}, we shall use Proposition \ref{prop:barrier} to demonstrate that $\rho_s$ produces sufficient damping so as to close the bootstrap argument. However, there is a technical caveat: $N_0$ is chosen according to initial datum $\rho_0$, while $\rho_s$ corresponds to initial datum $\rho(\calT_1,\cdot)$, which has a different $L^2$ norm from $\rho_0$. Therefore, applying Proposition \ref{prop:barrier} to $\rho_s$ will produce new thresholds, which we call $\widetilde{N_0}$, $\widetilde{T_*}$, and $\widetilde{g_0}$. Yet, we will demonstrate that one is still able to produce sufficient damping from $\rho_s$.} 
    
    Recall that $(\rho_s)_m = \rho_m$, and $\|\rho_s(0,\cdot) - \rho_m\|_{L^2}^2 = N_0$ by the definition of $\calT_1$. \red{Then by Proposition \ref{prop:barrier}, setting 
    $$
    \widetilde{N_0} = \max\left(1, 2\rho_m, 2N_0, \frac{32^3C_1\rho_m^4(15/8 + 2C_1c_0)}{4c_0}\right),
    $$
    $\widetilde{T_*} = 2c_0\widetilde{N_0}^{-1}$, there exists $\widetilde{g_0} = \widetilde{g_0}(\widetilde{N_0}, \rho_m)$ such that for any $g \ge \widetilde{g_0}$, we have}
    $
    \sup_{0 \le \tau \le \widetilde{T_*}}\|\rho_s(\tau) - \rho_m\|_{L^2}^2 \le \widetilde{N_0},
    $
    and there exists $\widetilde{T} \in (\widetilde{T_*}/2,\widetilde{T_*})$ such that
    $
    \|\rho_s(\widetilde{T}) - \rho_m\|_{L^2}^2 \le \frac{\widetilde{N_0}}{8}.
    $
    \red{An elementary but crucial observation is that $\widetilde{N_0} = 2N_0$: this is because we choose $N_0$ to satisfy \eqref{N0choice1} and \eqref{N0choice2}, from which we have
    $$
    N_0 \ge \max\left(1, 2\rho_m, \frac{32^3C_1\rho_m^4(15/8 + 2C_1c_0)}{4c_0}\right)
    $$
    and therefore $\widetilde{N_0} = 2N_0$.} Thus by definition, we also have $\widetilde{T_*} = \frac{T_*}{2}$.
    
    Now, we choose 
    \begin{equation}\label{g1choice}
        g_1 = g_1(\|\rho_0 - \rho_m\|_{L^2},\rho_m) := \widetilde{g_0}(\widetilde{N_0},\rho_m).
    \end{equation}
    With the discussions above, we conclude the following: assuming $g \ge g_1$, then there is 
    $$
    \|\rho_s(t) - \rho_m\|_{L^2}^2 \le \widetilde{N_0} = 2N_0
    $$ 
    for $t \in [0, T_*/2]$. Moreover, there exists $S_1 := \widetilde{T} \in [T_*/4, T_*/2]$ such that 
    $$
    \|\rho_s(S_1) - \rho_m\|_{L^2}^2 \le \frac{\widetilde{N_0}}{8} = \frac{N_0}{4}.
    $$ 
    
    \noindent\textbf{Step 3: Proof of Improved Estimate \eqref{est:improved}.} By triangle inequality and \eqref{est:rest}:
    \begin{align*}
    \sup_{\calT_1 \le t \le \calT_1 + T_*/2}\|\rho(t) - \rho_m\|_{L^2}^2 &\le 2\sup_{\calT_1 \le t \le \calT_1 + T_*/2}\|\rho_s(t-\calT_1) - \rho_m\|_{L^2}^2 + 2\sup_{\calT_1 \le t \le \calT_1 + T_*/2}\|r(t)\|_{L^2}^2\\
    &\le 4N_0 + \frac{N_0}{4} < 5N_0,
    \end{align*}
    and
    $$
    \|\rho(\calT_1 + S_1) - \rho_m\|_{L^2}^2 \le \frac{N_0}{2} + \frac{N_0}{4} < N_0.
    $$
    \red{Now, we may iterate this argument in a similar spirit to the proof of \eqref{est:l2dtu1} in Lemma \ref{lem:dtu}. For any $i > 1$, with $\calT_{i-1}$ and $S_{i-1} \in [T_*/4, T_*/2]$ defined, we consider $\calT_i$ to be the first time instance after $\calT_{i-1} + S_{i-1}$ such that $\|\rho - \rho_m\|_{L^2}^2 = N_0$. By definition of $\calT_i$, we have $\calT_i - \calT_{i-1} \ge S_{i-1} \ge T_*/4$. Moreover, we note that $\calT_i$ still verifies that $\calT_i > s_0$ and $\calT_i + T_* < \calT_0$. Indeed, the first inequality is due to the monotonicity of $\calT_i$ in $i$, and the second inequality holds true since $T_*$ is the least doubling time of $\|\rho - \rho_m\|_{L^2}^2$ at level $N_0$. So we may repeat Step 2 and 3 above on interval $[\calT_i, \calT_i + T_*]$:
    $$
    \sup_{\calT_i \le t \le \calT_i + T_*/2}\|\rho(t) - \rho_m\|_{L^2}^2 \le 5N_0,
    $$
    and
    $$
    \|\rho(\calT_i + S_i) - \rho_m\|_{L^2}^2 \le N_0,
    $$
    for some $S_i \in [T_*/4, T_*/2].$
    Iterating for $j_*:=\lfloor\frac{4\calT_0}{T_*}\rfloor + 1$ times, we reach the estimate that
    $$
    \sup_{0 \le t \le \calT_{j_*} + T_*/2}\|\rho(t) - \rho_m\|_{L^2}^2 \le 5N_0.
    $$
    However, we must have $\calT_{j_*} + T_*/2 \ge \calT_0$, since $\calT_{i+1} - \calT_i \ge T_*/4$ for every $i$. From this fact, we conclude that
    $$
    \sup_{0 \le t \le \calT_0}\|\rho(t) - \rho_m\|_{L^2}^2\le \sup_{0 \le t \le \calT_{j_*} + T_*/2}\|\rho(t) - \rho_m\|_{L^2}^2 \le 5N_0,
    $$
    which is the improved bound \eqref{est:improved}.
    }
    
\begin{rmk}\label{rmk:stokeschoice}
    We summarize the choice of thresholds $g_1$ and $B_1$ as follows. We first choose $g_1 = g_1(\rho_m, \|\rho_0 - \rho_m\|_{L^2})$ as in \eqref{g1choice}. Fixing $g \ge g_1$, we then choose $B_1 = B_1(\rho_m, \|\rho_0 - \rho_m\|_{L^2}, \|u_0\|_1)$ large so that \eqref{b1} and \eqref{b2} hold for all $B \ge B_1g^2e^{g^2}$.
\end{rmk}

{
\section{Discussions}
In this section, we discuss extensions of Theorem \ref{thm:stokeswp} and some open problems.
\begin{enumerate}
    \item By performing a similar, but less laborious, analysis, one could prove a similar result to Theorem \ref{thm:stokeswp} to the following Keller-Segel-Stokes system:
    \begin{equation}
        \label{eq:ksstokes1}
        \begin{cases}
        \p_t \rho + u\cdot \nabla \rho -\Delta \rho + \divv(\rho\nabla(-\Delta_N)^{-1}(\rho- \rho_m)) = 0,\\
        \rho_m = \frac{1}{|\Omega|}\int_\Omega \rho(t,x) dx,\\
        \p_t u - \frac{1}{\re}\Delta u + \nabla p = \ra\rho (0,1)^T,\\
        \divv u = 0.
        \end{cases}
    \end{equation}
    equipped with boundary conditions
    \begin{equation}
    \p_2 \rho = 0,\quad u_2  = 0,\quad \p_2u_1 = 0,\quad\quad \text{on $\p\Omega$}.
\end{equation}
The proof of this result was contained in an earlier ArXiv version of this work, but we only presented the Keller-Segel-Navier-Stokes case here due to considerable redundancies in the proof of the Keller-Segel-Stokes case.
\item It is natural to attempt to generalize our main result to the Keller-Segel-Navier-Stokes system equipping classical \textbf{no-slip} boundary condition on the Navier-Stokes equation. However, in this case, the existence of more singular boundary layers might complicate the analysis. Such effect is manifested in a different leading order problem. Instead of considering the Biot-Savart law $u = -g\nabla^\perp(-\Delta_D)^{-2}\p_1 \rho$ as we did in this work, one needs to use $u = \nabla^\perp \psi$, where $\psi$ solves the following bi-harmonic equation (see \cite{dalibard2023long,park2024stability,park2024long} for a more detailed discussion):
\begin{equation}
    \label{eq:biharmonic}
    \begin{cases}
        \Delta^2\psi = \p_1 \rho,\\
        \psi|_{\p\Omega} = \p_2\psi|_{\p\Omega} = 0.
    \end{cases}
\end{equation}
\end{enumerate}
}

%%%%%%%%%%%%%%%%%%%%%%%%%%%%%
\appendix
\section{Appendix}
\subsection{Stokes Operator Equipped with Lions Boundary Condition}
In this section, we first rigorously prove Proposition \ref{prop:Poincare} and Proposition \ref{prop:statstokes}, which concern Stokes operator corresponding to Lions boundary condition. For the convenience of the readers, we reiterate the statements of the propositions as follows:

\begin{prop}
\label{prop:a1}
    Let $u \in V$. Then the following estimate holds:
    \begin{equation}
    \label{est:a1}
    \|u\|_{L^2} \le C_P \|\nabla u\|_{L^2},
    \end{equation}
    where $C_P$ is a positive constant that only depends on the domain $\Omega$.
\end{prop}
\begin{proof}
    Since $\|u\|_{L^2}^2 = \|u_1\|_{L^2}^2 + \|u_2\|_{L^2}^2$, we will estimate each velocity component separately. Since $u \in V$, then in particular $u_1$ has zero spatial mean and $u_2 \in H^1_0(\Omega)$. Thus, \eqref{est:a1} follows from applying Poincar\'e inequality to $u_1$ and $u_2$ separately. 
\end{proof}

\begin{prop}
    \label{prop:a2}
    Assuming $f \in H$, then \eqref{eq:statstokes} admits a unique solution $u \in W$ with estimate
    \begin{equation}
        \label{est:elliptic1a}
        \|u\|_2 \le C\|f\|_{L^2}.
    \end{equation}
    In fact, $u$ is given by the following explicit formula:
    \begin{equation}
    \label{eq:bsorigina}
    u = \nabla^\perp(-\Delta_D)^{-2}(\p_2f_1 - \p_1f_2).
    \end{equation}
    More generally, if $f \in H^s \cap H$, $s \ge 1$, we have the following improved regularity estimate:
    \begin{equation}
        \label{est:elliptic2a}
        \|u\|_{s+2} \le C\|f\|_{s}.
    \end{equation}
\end{prop}
\begin{proof}
To show well-posedness of problem \eqref{eq:statstokes}, it is convenient to show on the level of stream function. Since we would find $u$ such that $\int_\Omega u_1 dx = 0$ and $u_2|_{\p\Omega} = 0$, there exists a unique stream function $\psi$ such that $u = \nabla^\perp\psi$ by Hodge decomposition. Moreover, $\psi$ satisfies the following:
    $$
    \Delta \psi = \omega,\quad \psi|_{\p\Omega} = 0.
    $$
    Taking $\nabla^\perp\cdot$ on \eqref{eq:statstokes} and denoting $\omega = \nabla^\perp \cdot u$, we must have
    $$
    -\Delta\omega = \p_1f_2 - \p_2f_1,\quad \omega|_{\p\Omega} = 0.
    $$
    Therefore, the discussion above motivates us to study the following auxiliary coupled system:

    \begin{equation}
    \begin{cases}
    \Delta \psi = \omega,\quad \psi|_{\p\Omega} = 0,\\
    -\Delta\omega = \p_1f_2 - \p_2f_1,\quad \omega|_{\p\Omega} = 0.
    \end{cases}
    \end{equation}

    First, note that since $f \in L^2$, we have $\p_1f_2 - \p_2 f_1  \in H^{-1}_0$ in the sense of distribution. Then by standard elliptic theory, there exists a unique $\omega \in H^1_0$ such that the $\omega$ equation holds, with estimate $\|\omega\|_1 \le C(\Omega) \|f\|_{L^2}$ . Applying standard elliptic theory again to $\psi$ equation, we have a unique solution $\psi \in H^3 \cap H^1_0$, with estimate $\|\psi\|_3 \le C(\Omega)\|\omega\|_1$. Define $u = \nabla^\perp \psi$. Then $u$ strongly solves \eqref{eq:statstokes} with boundary conditions $u_2|_{\p\Omega} = \p_2u_1|_{\p \Omega} = 0$ satisfied in the trace sense. Moreover, we have the estimate
    $$
    \|u\|_2 \le \|\psi\|_{3} \le C(\Omega)\|f\|_{L^2}.
    $$
    Moreover from the discussion above, we obtain formula \eqref{eq:bsorigina}. Finally, to show the higher regularity estimates \eqref{est:elliptic2a}, we note that given $f \in H^s \cap H$, $s \ge 1$, $\nabla f \in H^{s-1}$. Using the explicit formula \eqref{eq:bsorigina} and standard elliptic estimate, we conclude \eqref{est:elliptic2a}.
\end{proof}

Finally, we present the following calculus lemma:
    \begin{lem}
    \label{lem:ibp}
        $\|\nabla^2 u\|_{L^2} = \|\Delta u\|_{L^2}$ for $u \in W \cap C^\infty(\Omega)$.
    \end{lem}
    \begin{proof}
        Integrating by parts twice, we obtain that
    \begin{align*}
        \|\Delta u\|_{L^2}^2 &= \int_\Omega \p_i\p_iu_j\p_k\p_ku_j dx\\
        &= \int_\Omega \p_i\p_k u_j \p_i\p_ku_j dx + \int_{\p\Omega}\Delta u_j\p_2 u_j dx_1 - \int_{\p\Omega} \p_k\p_2u_j \p_k u_j dx_1\\
        &=: \|\nabla^2 u\|_{L^2}^2 + BT_1 + BT_2.
    \end{align*}
    We moreover notice that
    \begin{align*}
        BT_1 &= \int_{\p\Omega} \Delta u_1 \p_2u_1 dx_1 + \Delta u_2 \p_2u_2 dx_1 = 0,
    \end{align*}
    Now we observe that for any $x \in \p\Omega$, since $u$ is smooth and belongs to the class $W$, we have
    $$
    \Delta u_2 = \p_1^2u_2 + \p_2^2 u_2 = \p_2^2 u_2 = -\p_1\p_2u_1 = 0,
    $$
    where we used $u_2|_{\p\Omega} = 0$ in the second equality, $\divv u = 0$ in the third equality, and $\p_2 u_1|_{\p\Omega} = 0$ in the last equality. Using this fact and $\p_2 u_1 = 0$ for $x \in \p\Omega$, we expand the sum and obtain:
    \begin{align*}
        BT_2 &= \int_{\p\Omega} \p_1\p_2u_1 \p_1u_1 dx_1 + \int_{\p\Omega} \p_1\p_2u_2 \p_1u_2 dx_1\\
        &+\int_{\p\Omega} \p_2\p_2u_1 \p_2u_1 dx_1 + \int_{\p\Omega} \p_2\p_2u_2 \p_2u_2 dx_1\\
        &= \int_{\p\Omega} \p_2\p_2u_2 \p_2u_2 dx_1,
    \end{align*}
    where we used the boundary conditions of $u$ in the last equality. Since $\p_2^2 u_2 = -\p_1\p_2 u_1 = 0$ for $x \in \p\Omega$, we finally conclude that $BT_2 = 0$, and thus $\|\nabla^2 u\|_{L^2} = \|\Delta u\|_{L^2}$.
    \end{proof}

\end{document}